\documentclass[10pt]{article}
\usepackage{amssymb,amsfonts}
\usepackage{pb-diagram}
\usepackage{amsmath,amsthm,amsfonts,amssymb}
\usepackage[usenames]{color}
\usepackage{hyperref}
\usepackage{soul}
\usepackage{graphicx}
\newtheorem{theorem}{Theorem}
\newtheorem{definition}{Definition}

\newtheorem{lemma}{Lemma}
\newtheorem{cor}{Corollary}

\newtheorem{prop}{Proposition}

\newcommand{\be}{\begin{enumerate}}
\newcommand{\ee}{\end{enumerate}}
\newcommand{\beq}{\begin{equation}}
\newcommand{\eeq}{\end{equation}}
\usepackage{amssymb}
\usepackage{amsmath}
\usepackage{amssymb}

\newenvironment{romanenumerate}
	{\begin{enumerate}

	}
        {
        
        \end{enumerate}
        } 

\newcommand{\comment}[1]{} 

\begin{document}

\title{Quantifier elimination algorithm  to boolean combination of $\exists\forall$-formulas in the theory of a free group }
\author{Olga Kharlampovich, Alexei Myasnikov}

\maketitle
\begin{abstract}
It was proved by Sela and by the authors that every formula in the theory of a free group $F$ is equivalent to a boolean combination of $\exists\forall$-formulas. We also proved that the elementary theory of a free group is decidable (there is an algorithm given a sentence to decide whether this sentence belongs to $Th(F)$ ). In this paper we give an algorithm for reduction of a first order formula over a free group to the equivalent boolean combination of $\exists\forall$-formulas.
\end{abstract}
\section{Introduction}
It was proved in \cite{Sela},\cite{KMel} that every formula in the theory of a free group $F$ is equivalent to a boolean combination of $\exists\forall$-formulas.  We also proved that the elementary theory of a free group is decidable (there is an algorithm given a sentence to decide whether this sentence belongs to $Th(F)$ ).
If the language of the free group $F$ contains constants, then it was shown in \cite{KMdef} that
 every definable subset of $F$ is defined by some boolean combination of
formulas \begin{equation}\label{AE}
\exists X\forall Y (U(P,X)=1\wedge V(P,X,Y)\neq 1),
\end{equation}where $X,Y,P$ are tuples of variables. We call these formulas {\em conjunctive $\exists\forall$-formulas}.
We will prove the following result.
\begin{theorem} \label{Th1} \label{main} Let $F$ be a free  group. There exists an algorithm given a first-order formula $\phi$  to find a boolean combination of $\exists\forall$-formulas that defines the same set as $\phi$ over $F$.
\end{theorem}

\section{Preliminary results}This result was not explicitly formulated in \cite{KMel}.  We also corrected some inaccuracies and errors  in \cite{KMel}. The reader should have \cite{KMel}, Section 8 as a reference for  most of  the required algorithmic results.

\subsection{Effectiveness of the relative JSJ decomposition}
In \cite{JSJ} we proved the following results.
\begin{theorem}\cite{JSJ}\label{th:JSJ} There exists an algorithm to obtain a cyclic
and abelian $JSJ$ decompositions for a f.g. fully residually free
group relative to the subgroup of constants $F$ (meaning that $F$ is elliptic in the decomposition) and to find the maximal standard quotient (see \cite{KMel}, Definition 17).
\end{theorem}


\begin{theorem} \label{modulo} (\cite{JSJ}, Theorem 13.1, \cite{KMel}, Theorem 35)  There exists an algorithm to obtain an abelian $JSJ$
decomposition for a f.g. fully residually free group relative to finitely generated
subgroups $K_1,\ldots ,K_m$ and to find the maximal standard quotient. \end{theorem}

\subsection{Encoding solutions with the tree ${T}$.}\label{can_diagram}
An algorithm is described in Section 7 of  \cite{KMel} which constructs, for a given system of equations $S(X,A)$ over the free group $F$,
a canonical $Hom$-diagram encoding the set of solutions of $S$. The diagram consists of a directed finite rooted tree $T$ (the augmented canonical embedding tree, see \cite{KMel}, Section 7.6) that has, in particular,
the following properties.  Let $G=F_{R(S)}$.
\begin{romanenumerate}
\item Each vertex $v$ of $T$ is labelled by a pair  $(G_{v},Q_{v})$, where $G_{v}$ is an $F$-quotient of $G$ and $Q_{v}$ the subgroup of canonical automorphisms in $\mathrm{Aut}_{F}(G_{v})$ (see \cite{KMel} for the definition).
The root $v_0$ is labelled by $(G,1)$ and every leaf is labelled by $(F(Y)\ast F,1)$, where $Y$ is some finite set (called \emph{free variables}).
Each $G_{v}$, except possibly $G_{v_{0}}=G$, is fully residually $F$.
\item Every (directed) edge $v\rightarrow v'$ is labelled by a proper surjective $F$-homomorphism $\pi(v,v'):G_{v}\rightarrow G_{v'}$.
\item For every $\phi\in\mathrm{Hom}_{F}(G,F)$ there is a path $p=v_0 v_1 \ldots v_k$ where $v_k$ is a leaf labelled by
$(F(Y)\ast F,1)$, elements $\sigma_{i}\in Q_{v_{i}}$, and a $F$-homomorphism
$\phi_{0}: F(Y)\ast F\rightarrow F$ such that
\begin{equation}\label{8}
\phi = \pi(v_0,v_1) \sigma_1 \pi(v_1,v_2) \sigma_2 \cdots \pi(v_{k-2},v_{k-1})\sigma_{k-1}\pi(v_{k-1},v_{k})\phi_{0}.
\end{equation}
\item The canonical splitting of each fully residually free group $G_{v}$ is its Grushko decomposition followed by the abelian JSJ decompositions of the factors.
\end{romanenumerate}
The set of all homomorphisms in (\ref{8}) associated with a path from the root to a leaf of the tree $T$ is called a {\em fundamental set} or {\em fundamental sequence} of solutions.

The algorithm gives for each $G_{v}$ a finite presentation $\langle A_{v}|\mathcal{R}_{v}\rangle$, and for each $Q_{v}$ a finite list of
generators in the form of functions $A_{v}\rightarrow (A_{v}\cup A_{v}^{-1})^{*}$.  Note that the choices for $\phi_{0}$ are
exactly parametrized by the set of functions from $Y$ to $F$.

\begin{definition}
\begin{enumerate} A fundamental sequence corresponding to the branch
$$G\rightarrow _{\pi_0} G_1\rightarrow _{\pi_1}  G_2\rightarrow\ldots \rightarrow _{\pi_{k-1}}  G_k=F(Y)*F$$

 is called {\em strict}  if it has the following properties:
\item The image of each non-abelian vertex group of $G_i, 1\leq i< k,$ under $\pi _i$ is non-abelian.
\item For each $1\leq i< k$, $\pi_i$ is injective on rigid subgroups, edge groups, and subgroups generated by the images of edge groups in abelian vertex groups in $G_{i-1}$.
\item For each $1\leq i <k $, if $R$ is a rigid subgroup in $G_{i}$ and $\{A_j\}$, $1\leq j\leq m$, the abelian vertex groups in $G_{i}$ connected to $R$ by edge groups $E_j$ with the maps $\eta_j:E_j\rightarrow A_j$, then $\pi_i$ is injective on the subgroup $\langle R,Is(\eta_1(E_1)),\ldots,Is(\eta_m(E_m))\rangle , $ where
$Is(E_i)$ is the isolator of $E_i$ in $A_i$ (the minimal direct factor
containing $E_i$) which we will call the {\em envelop of $R$}.
\item The images of different factors in the Grushko decomposition of $G_i$ under $\pi _i$  are different factors in the free decomposition of $G_{i+1}$. 
\end{enumerate}
\end{definition}

{\bf Definition of a canonical NTQ system of equations and NTQ group.} The definition of an NTQ system can be found, for example, in \cite{KMel}, Section 2.3. We can associate a canonical NTQ system (and an NTQ group that is the coordinate group of the system) to the strict fundamental sequence above as in \cite{KMel}, Section 7.3. We now recall the construction here.

We do this inductively beginning with the bottom group $G_k$. 
Each free factor of $G_{k-1}$ (denote it $\bar G_{k-1}$) is mapped by $\pi _{k-1}$ to a product of free factors of $G_{k}$ (denote it by $\bar G_{k}$). Let $D$ be a JSJ decomposition of  $\bar G_{k-1}$.  

Step 1. For each edge group $E$ in $D$ between rigid subgroups we extend the centralizer of $\pi_{k-1}(E)$ in $\bar G_k$ by a new letter $y_E$ and add commutativity relations that make the group $H_1$ obtained from $G_{k}$ by these centralizer extensions commutative transitive.  The fundamental group $H_0$ of the graph of groups corresponding to rigid subgroups in $D$ and edges between them is embedded
in $H_1$.

Step 2. Let $A$ be a non-cyclic abelian vertex group in $D$ and
$A_e$ the subgroup of $A$ generated by the images in $A$ of the edge
groups of edges adjacent to $A$. Then $A = Is(A_e) \times A_0$ where
$Is(A_e)$ is the isolator of $A_e$ in $A$  and $A_0$ a  direct complement of $Is(A_e)$ in
$A$. Notice, that the restriction of $\pi _{k-1}$ on $Is(A_e)$ is a
monomorphism (since $\pi_{k-1}$ is injective on $A_e$ and $A_e$ is of
finite index in $Is(A_e)$). For each non-cyclic abelian vertex group
$A$ in $D$ we extend the centralizer of $\pi _{k-1}(Is(A_e))$ in $H_1$
by the abelian group $A_0$. 

If  $\pi _{k-1}(A_e)$ and  $\pi _{k-1}(A_{e_1})$ happen to be conjugate in $H_1$  into the same abelian subgroup, then we add commutativity relations to make the group commutative transitive and will have one abelian vertex group instead of these two. We denote the resulting group by
$H_2$. 

 If the  abelian group $A_0$ has rank
$r$ then the system of equations associated with the abelian vertex
group $A$ has the following form
\begin{equation}\label{eq:barS} [y_p, y_q]=1, [y_p, {\bar d_{ej}}]=1, \ \
 \  p,q=1,\ldots ,r, j = 1, \ldots, p_e,\end{equation}
 where $y_p, y_q$ are new variables and the elements ${\bar d_{e1}}, \ldots, {\bar
d_{ep_e}}$
 are constants from $H_1$  which generate  the subgroup $\pi _{k-1}(Is(A_e))$.
 We assume that the constants ${\bar d_{ej}}$
 are given as words in the generators $g_1, \ldots, g_l$ of $\bar G_{k}$.

Step 3.  
Let $Q$ be a QH subgroup in $D$. Suppose
$Q$ is given by a presentation
$$\prod _{i=1}^{n}[x_i, y_i]p_1\cdots p_{m}=1.$$
 where  there are exactly $m$
outgoing edges $e_1,\ldots ,e_m$ from $Q$  and $\sigma(G_{e_i})=
\langle p_i\rangle $,  $\tau(G_{e_i}) = \langle c_i\rangle $ for
each edge $e_i$. We add a QH vertex $Q$ to $H_2$ by introducing
new generators and the following quadratic relation

\begin{equation}\label{QQQC}\prod _{i=1}^{n}[ x_i,
y_i](c_1^{\pi_{k-1}})^{z_1}\cdots
(c_{m-1}^{\pi_{k-1}})^{z_{m-1}}c_m^{\pi_{k-1}}=1\end{equation}
 to the presentation of $H_2$.
Observe, that in the relations (\ref{QQQC}) the coefficients in the
original quadratic relations for $Q$ in $D$ are
 replaced by their images in $\bar G_k$.
The resulting group is denoted by $H_2$.

We construct such a group for each free factor of $G_{k-1}$ and take their free product $H$. Then $H$ is the NTQ group corresponding to level $k-1$ of the fundamental sequence. Similarly, taking $G_{k-2}$ and $H$ instead of $G_{k-1}$ and $G_k$, we construct the NTQ group corresponding to level $k-2$ and so on up to level 1.

\subsection{Generic Families}\label{subsection:gen-fam}

To detect splittings of limit groups we will need  the notion of a {\em generic family} of solutions of an NTQ system in a free (or a torsion-free hyperbolic) group $G$. It is given in \cite{KMel} (Definition 22) and is very technical. To make this paper more self contained we will give a definition here, we also give it in the language of \cite{Sela2}, Definition 1.5.   We will describe the construction of particular families of solutions, called generic families, of an NTQ system which help  to detect splittings and also  imply nice lifting properties for that system.  But first we explain what kind of lifting properties we wish to have.

Let $S(X) = 1$ be a system of equations with a solution in a torsion free hyperbolic group $G$. We say that a system of equations $T(X,Y)
= 1$ is {\em compatible} with $S(X) = 1$ over $G$, if for every
solution $\bar a$ of $S(X) = 1$ in $G$, the equation $T(\bar a,Y) = 1$ also
has a solution in $G$. More generally, a formula $\Phi(X,Y)$ in
the language $L_A$ is {\em compatible} with $S(X) = 1$ over $G$,
if for every solution $\bar{a}$ of $S(X) = 1$ in $G$ there exists
a tuple $\bar{b}$ over $G$ such that the formula $\Phi(\bar a,
\bar b )$ is true in $G$.

Suppose now that a formula $\Phi(X,Y)$ is compatible with $S(X)=
1 $ over $G$. We say that $\Phi(X,Y)$ {\em admits a lift to a
generic point} of $S(X) = 1$ over $G$ (or just that it has an {\em $S$-lift} over
$G$), if the formula $\exists Y \Phi(X,Y)$ is true in the coordinate group
$G_{R(S)}$ (here $Y$ are variables and $X$ are constants from
$G_{R(S)}$). Finally, an equation $T(X,Y) = 1$, which is
compatible with $S(X) = 1$, admits a {\em complete $S$-lift} if
every formula $T(X,Y) = 1 \ \& \ W(X,Y) \neq 1$, which is
compatible with $S(X) = 1$ over $G$, admits an $S$-lift. We say
that the lift (complete lift) is {\em effective} if there is an
algorithm to decide for any equation $T(X,Y)=1$ (any formula
$T(X,Y) = 1 \ \& \ W(X,Y) \neq 1$) whether $T(X,Y)=1$ (the formula
$T(X,Y) = 1 \ \& \ W(X,Y) \neq 1$) admits an $S$-lift, and if it
does, to construct a solution in $G_{R(S)}.$

We now decribe the construction of generic families of solutions of an NTQ system. 
Consider a fundamental sequence with corresponding NTQ system $S(X,A)=1$ of depth $N$ over a torsion free hyperbolic group $G$. We construct generic families iteratively for each level $k$ of the system, starting at $k=N$ and decreasing $k$. There is an abelian decomposition of $G_k$ corresponding to the NTQ structure. Let $V_1^{(k)},\ldots,V_{M_k}^{(k)}$ be the vertex groups of this decomposition given in some arbitrary order. We construct a generic family for level $k$, denoted $\Psi(k)$, by constructing generic families for each vertex group in order. We denote a generic family for the vertex group $V_i^{(k)}$ by $\Psi(V_i^{(k)})$. If there are no vertex groups, in other words the equation $S_k=1$ is empty ($G_k=G_{k+1}*F(X_k)$) we take $\Psi(k)$ to be a sequence of growing different Merzljakov's words (defined in \cite{Imp}, Section 4.4).

If $V_r^{(k)}$ is an abelian group then it corresponds to equations of the form $[x_i,x_j]=1$ or $[x_i,u]=1$, $1\leq i,j\leq s$,  where $u\in U$ runs through generators of a centralizer in $G_{k+1}$. A solution $\sigma$ in $G_{k+1}$ to equations in these forms is called \emph{$B$-large} if there are some $b_1,\ldots,b_s$ such that for each $i$, $b_i> B$ and such that $\sigma(x_i)=(\sigma(x_1))^{b_1\ldots b_i}$ or $\sigma(x_i)=u^{b_1\ldots b_i}$, for $1\leq i\leq s$ (possibly renaming $x_1$). A generic family of solutions for an abelian subgroup $V_r^{(k)}$ is a family $\Psi(V_r^{(k)})$ such that for each $B_i$ in any increasing sequence of positive integers $\{B_i\}_{i=1}^{\infty}$ there is a solution in $\Psi(V_r^{(k)})$ which is $B_i$-large.

If  $V_r^{(k)}$ is a QH vertex group of this decomposition, let $S$ be the surface associated to  $V_r^{(k)}$. We associate two collections of non-homotopic, non-boundary parallel, simple closed curves $\{b_1,\ldots b_q\}$ and $\{d_1,\ldots d_t\}$. These collections should have the propery that $S-\{b_1\cup\cdots\cup b_q\}$ is a disjoint union of three-punctured spheres and one-punctured Mobius bands, each of the curves $d_i$ intersects at least one of the curves $b_j$ non-trivially, and their union fills the surface $S$ (meaning that the collection $\{b_1,\ldots b_q, d_1,\ldots ,d_t\}$ has minimal number of  intersections and $S-\{b_1,\cup\cdots\cup b_q\cup d_1\cup\cdots\cup d_t\}$ is a union of topological disks).

Let $\beta _1,\ldots ,\beta_q$ be automorphisms of $V_r^{(k)}$ that correspond to Dehn twists along $b_1,\ldots b_q$ , and $\delta _1,\ldots ,\delta_t$ be automorphisms of $V_r^{(k)}$ that correspond to Dehn twists along $d_1,\ldots d_t$.
We define iteratively a basic sequence of automorphisms $\{\gamma _{L,n}, \phi _{L,n}\}$ (compare with Section 7.1 of \cite{Imp} where one particular basic sequence of automorphisms is used), which is determined by a sequence of  ($t+q$)-tuples $L=\{(p_{1,n},\dots,p_{t,n},m_{1,n},\ldots,m_{q,n})\}_{n=1}^{\infty}$

Let $$\phi _{L,0}=1$$

$$\gamma _{L,n}=\phi_{L,n-1}\delta _1^{m_{1,n}}\ldots \delta _q^{m_{q,n}}, n\geq 1$$

$$\phi _{L,n}=\gamma_{n}\beta _1^{p_{1,n}}\ldots \beta _t^{p_{t,n}},n\geq 1$$

Assuming that generic families have already been constructed for $V_i^{(k)}$, $i< r$, and for every vertex group in levels $k'>k$, and that $\Theta _k$ is a family of growing powers of Dehn twists for edges on level $k$, set $\Psi(k')=\Psi(V_{M_{k'}}^{(k')})\Theta _{k'}$ for $k<k'\leq N$ (in other words the generic family for level $k'$ is the generic family of the last vertex group at that level) and set $\Psi(N+1)=\{1\}$. Let $\pi_k:G_k\rightarrow G_{k+1}$ the canonical epimorphism. Let $\Sigma_r^{(k)}=\{\psi_{1}\cdots\psi_{r-1}|\psi_{i}\in\Psi(V_i^{(k)})\}$ be the collection of all compositions of generic solutions for previous vertex groups. We then say that $$\Psi(V_r^{(k)})=\{\mu _{L,n,\lambda_n}=\phi _{L,n}\delta _1^{\lambda_n}\ldots \delta _q^{\lambda_n}\sigma_n\pi_k\tau |\sigma_n\in\Sigma_r^{(k)}, \tau\in\Psi(k+1)\}_{n=1}^{\infty}$$ where each $\lambda_n$ is some positive integer, is a generic family for $V_r^{(k)}$ if it has the following property: Given any $n$ and any tuple of positive numbers $\overrightarrow{A}=(A_1,\ldots,A_{nt+nq+1})$ with $A_i< A_j$ for $i< j$, $\Psi$ contains a homomorphism $\mu_{n,L,\lambda_n}$ such that the tuple

\begin{align*}
\overrightarrow{L}_{n,r_n} &=(p_{1,1},\ldots ,p_{t,1}, m_{1,2},\ldots ,m_{q,2},\ldots, m_{1,n},\ldots ,m_{q,n}, p_{1,n},\ldots ,p_{t,n}, \lambda_n) \\
 &= (L_1,\ldots,L_{nt+nq+1})
\end{align*}grows faster than $\overrightarrow{A}$ in the sense that $L_1\geq A_1$ and $L_{i+1}-L_{i}\geq A_{i+1}-A_{i}$.

Finally we set $\Psi(S)=\Psi(V_{M_1}^{(1)})$ to be a generic family of solutions for the $G$-NTQ system $S(X,A)=1$. Notice $\Psi(S)$ discriminates $G_{R(S)}$. 

We proved in \cite{Imp} the following two facts. 
\begin{prop} \label{IFT} If $\Psi(W)$ is a generic family of solutions for a regular
NTQ-system $W(X,A)=1$ over a free group $F$, then for any equation
$V(X,Y,A)=1$ the following is true: if for any solution $\psi\in\Psi(W)$
there exists a solution of $V(X^{\psi},Y,A)=1$, then $V=1$ admits a
complete $W$-lift. 

If the
NTQ-system $W(X,A)=1$ is not regular, then for any equation
$V(X,Y,A)=1$ the following is true: if for any solution $\psi\in\Psi(W)$
there exists a solution of $V(X^{\psi},Y,A)=1$, then $V=1$ admits a
complete $W$-lift into corrective extensions of
$W(X,A)=1$ (see \cite{Imp} for the definition of a corrective extension). There is a finite number of these extensions and any solution of $W(X,A)=1$ factors through one of them.
\end{prop}

\section{Decision algorithm for $\forall\exists$-sentences}\label{sec:4}
 In the rest of the paper we will only consider fundamental sequences satisfying the first and second restrictions from
 \cite{KMel}, Sections 7.8, 7.9. To make this paper self-contained we recall these sections here.

 \subsection{First restriction on fundamental sequences}
 We only consider strict fundamental sequences. Recall, that, in particular, this means that we include in the fundamental sequence only such homomorphisms that
give nonabelian images of the regular subsystems of ${Q_i}=1$ on
all  levels, and the images
of the edge groups of the JSJ decompositions on all levels are
nontrivial. 
\label{spl}  The  {\em first restriction} on fundamental sequences is the following property.

Let $S(Z)=1$ be a system of equations over $F$. We can assume that it is irreducible. We construct fundamental sequences for $S(Z,A)=1$ over $F$ as in \cite{KMel}, Section 7.7.  Let
\begin{equation}\label{3.5.} F _{R(\bar U)}*F(t_1,\ldots
,t_k)=P_1*\ldots *P_q*\langle t_1\rangle *\ldots *\langle t_k\rangle \end{equation} be a reduced
 free decomposition of a maximal shortening quotient $F _{R(\bar U)}*F(t_1,\ldots
,t_k)$ modulo $F$ for ${F}_{R(S)}$ (this shortening quotient is exactly the limit group corresponding to the second from the top level  of the corresponding fundamental sequence), and $\pi : F _{R(S)}\rightarrow F _{R(
\bar U)}\ast F(t_1,\ldots ,t_{\beta _i})$ the  epimorphism.  Let $Q$ be an MQH subgroup in the
JSJ decomposition of  a free factor of ${F}_{R(S)}.$ Consider a free decomposition $\pi
(Q)=K_1*\ldots
*K_p*\langle t_{k_{j_1}}\rangle *\ldots *\langle t_{k_{j_2}}\rangle $ inherited from the free
decomposition (\ref {3.5.}) such that each of the images of the boundary elements of $Q$  can be 
conjugated into some $K_j$, and each $K_j$ has a conjugate of an image of some
boundary element.

 We cut the punctured surface $\Sigma$ corresponding to the MQH subgroup $Q$  along a  maximal collection of disjoint non-homotopic simple closed curves that corresponds to the lift of the free decomposition $\pi
(Q)=K_1*\ldots
*K_p*\langle t_{k_{j_1}}\rangle *\ldots *\langle t_{k_{j_2}}\rangle $ and are  mapped to the identity by $\pi$. Moreover, if $\Sigma '$ is a punctured surface obtained from $\Sigma$ and connected to  a rigid subgroup, then when we   adjoin disks to the  boundary components of  $\Sigma '$  that are mapped to the identity, no s.c.c. (simple closed curve) on that surface is mapped to the identity by $\pi$.  If $\Sigma '$ 
 is a punctured surface obtained from $\Sigma$ and not connected to  a rigid subgroup, then when we   adjoin disks to the  boundary components of  $\Sigma '$  that are mapped to the identity, we obtain a closed surface, and $\pi$ maps it onto a free group of maximal possible rank.
 
We require a similar property for all MQH subgroups and for all levels of the fundamental sequence.

\subsection{Second restriction on fundamental sequences} \label{rk3}

Suppose the family of homomorphisms $\sigma _1\pi _1\ldots \sigma
_n\pi _n\phi _0$ is a strict fundamental sequence, corresponding to the NTQ
system $Q(X_1,\ldots, X_n)=1:$
$$Q_1(X_1,\ldots ,X_n)=1,$$
$$\ldots $$
$$Q_n(X_n)=1$$
adjoint with free variables $t_1,\ldots, t_k$. Here the restriction of $\sigma _i$
on $F _{R(Q_i,\ldots,Q_n)}$ is the canonical automorphism on
$F _{R(Q_i,\ldots,Q_n)}$,  identical on variables from
$X_{i+1},\ldots ,X_n$ and on all free variables from the higher
levels, $\pi _i:F _{R(Q_i,\ldots,Q_n)}*F(t_1,\ldots
,t_{k_{i-1}})\rightarrow F _{R(Q_{i+1},\ldots,Q_n)}*F(t_1,\ldots
,t_{k_i}).$  The {\em dimension of the fundamental sequence} is the sum $k_1+\ldots +k_n.$

We can  suppose that all fundamental sequences that we consider
satisfy the following properties. Let $F _{R(Q_i,\ldots ,Q_n)}$ be a
free product of some factors. Then
  \begin {enumerate}
\item [1)] the images of abelian factors   under $\pi _i$  are
different factors of $F(t_{k_{i-1}+1},\ldots ,t_{k_i})$; \item [2)]
the images under $\pi _i$ of factors which are surface groups are
different  factors of $F(t_{k_{i-1}+1},\ldots ,t_{k_i})$;
\item [3)] if some quadratic equation in $Q_i=1$ has free
variables in this fundamental sequence, then these variables
correspond to some variables among $t_{k_{i-1}+1},\ldots ,t_{k_i}$,
the images under $\pi _i$ of coefficients of quadratic equations
cannot be conjugated into $F(t_{k_{i-1}+1},\ldots ,t_{k_i})$; \item
[4)] different  factors in the free decomposition of
$F _{R(Q_i,\ldots ,Q_n)}$ are sent into different factors in the free
decomposition of $F _{R(Q_{i+1},\ldots
,Q_n)}*F(t_{k_{i-1}+1},\ldots,t_{k_i}).$
\end{enumerate}

\begin{prop}\cite{KMel}
\label{pr:solutions-special-Sb}  For a system of equations $S$ over $F$ one can effectively construct a
finite set of  strict fundamental sequences  that  satisfy  the first and second
restriction above, such that every solution of $S$ factors through one of these fundamental sequences. These fundamental sequences correspond to a tree which we denote $T_{CE}(F _{R(S)})$ (canonical embedding tree).
\end{prop}

\begin{definition} We call  fundamental sequences satisfying the first and second restrictions {\em well aligned} fundamental sequences.\end{definition}

\subsection{Induced NTQ systems and fundamental sequences}\label{inher}

In this subsection we describe the construction of \cite{KMel}, Section 7.12 . Given an NTQ system $S=1$, the corresponding NTQ group $F _{R(S)}$, and the well aligned fundamental sequence of solutions, we will construct the induced NTQ group, the NTQ system, and the well aligned fundamental sequence of solutions for a subgroup $K$ of $F _{R(S)}$.

Let $S=1$ be an NTQ system over $F$:

\medskip
$S_1(X_1, X_2, \ldots, X_n,A) = 1,$

\medskip
$\ \ \ \ \ S_2(X_2, \ldots, X_n,A) = 1,$

$\ \ \ \ \ \ \ \ \ \  \ldots$

\medskip
$\ \ \ \ \ \ \ \ \ \ \ \ \ \ \ \ S_n(X_n,A) = 1$

\medskip \noindent
and $\pi _i: F_i \rightarrow F_{i+1}$   a fixed
$F_{i+1}$-homomorphism (a solution of $S_i(X_1,\ldots ,X_n)=1$ in
$F_{i+1} = {F}_{R(S_{i+1},\ldots ,S_n)},$ $F_{n+1} = G$, which is a free product of $F$ and a  free group). 
Let $K$ be a finitely generated subgroup (or $F$-subgroup) of
$F_{R(S)}.$ Then there exists a system $W(Y) = 1$ such that  $ K =
F_{R(W)}.$  We will describe here how to embed $K$ more economically into an
NTQ group $F_{R(Q)}$ such that ${F_{R(Q)}}\leq F_{R(S)}$ and assign to
$Q=1$ a fundamental sequence that includes all the solutions of
$W=1$ that factor through the  well aligned  fundamental sequence for $S=1$ that we started with.

Canonical automorphisms on different levels for $Q=1$ will be
induced by canonical automorphisms for $S=1$, mappings between
different levels for $Q=1$ will be induced by mappings for $S=1.$
 
 Without loss of generality we can suppose
that $F_{R(S)}$ is freely indecomposable modulo $F$.  The top
 system of equations $S_1(X_1,\ldots ,X_n)=1$ corresponds
to a splitting $D$ of $F_{R(S)}$. The non-QH non-abelian vertex subgroups
of $D$ are factors in a free decomposition of $\langle X_2,\ldots ,X_n\rangle .$
Consider the induced splitting of  $K$ denoted by $D_K$. This
splitting may give a free factorization $K= K_1*\ldots *K_k,$ where
$F\leq  K_1$. Consider each factor separately.  Consider $K_1$. Each edge $e$ in the decomposition of $K_1$ (induced by $D_K$ on the free factor $K_1$) that connects two rigid vertex groups is composed from  two edges $e_1$ and $e_2$ that are adjacent and both are in the orbit of the same edge $\bar e$ in the Bass-Serre tree corresponding to the graph of groups $D$. Moreover, $\bar e$ connects a rigid vertex group to an abelian vertex group in $D$, because if it connects 
a rigid vertex group to a QH-subgroup, then  a QH subgroup would appear between two rigid vertex groups instead of edge $e$. Increasing $ K_1$ by
a finite number of suitable elements from abelian  vertex groups of
$F_{R(S)}$ we join together non-QH non-abelian subgroups of $D_K$
which are conjugated into the same non-QH non-abelian subgroup of
$D$ by elements from abelian vertex groups in $F_{R(S)}.$  Moreover, we can choose these elements in abelian subgroups in such a way that their images on the next level are trivial. Indeed, each abelian subgroup $A$ is the direct product of the isolator of the edge group $A_1$ (in a primary decomposition this isolator is $A_1$ itself) and a free abelian subgroup $A_2$. This conjugating element $a$ belongs to $A_2$ and, therefore, is mapped on the next step to (the image of) $A_1$, say $a_1\in A_1$. Then $a_1^{-1}a$ is the desired element that is mapped to the identity. When we add this element to $K_q,$ instead of the  two non-QH non-abelian subgroups in $D_k$ considered above we have one non-QH non-abelian subgroup. This way we obtain a group $\bar K_1$ such that $D_{\bar K_1}$ does not
have edges between non-QH non-abelian subgroups, and generators of
edge groups connecting non-QH non-abelian subgroups to abelian
subgroups not having roots in $\langle X_2,\ldots ,X_n\rangle .$  We do the same if an abelian vertex group is connected to two rigid subgroups. Denote the group that we obtained by $\hat K_1$.

Since we will be considering only well aligned fundamental sequences,  we fix the family of s.c.c. in the QH subgroups of $\hat K_1$ that are mapped to the identity by $\pi _1$, so that the corresponding quadratic equations split into systems satisfying  the first restriction. This corresponds to a collection of s.c.c. which are the pre-images of the collection of s.c.c. in
a QH subgroup of $F_{R(S)}.$ We will add conjugating elements so that each punctured surface connected to a rigid subgroup in the  decomposition refined by splittings along these s.c.c, is connected to a unique rigid subgroup.  Now all rigid subgroups that are mapped to the same free factor on the next level of ${F}_{R(S)}$, are conjugate into one subgroup.
The images of these elements in $F _{R(S)}$ are mapped to the identity by $\pi _1$. Denote the obtained group by $\tilde K_1$.  We have  $\pi _1(K)=\pi _1(\tilde K_1).$

Now
consider separately each factor in the free decomposition of $\pi _1(K)\cap \langle X_2,\ldots
,X_n\rangle $ and enlarge it the same way. Working similarly with each
$\Gamma_i$ we consider all the levels of $S=1$ from the top to the bottom. 
We  obtain a subgroup generated by $\tilde K_1,\ldots , \tilde K_k$ and the enlarged images of them on all the levels, and   denote it by $H_1$. 

Then we repeat the whole
construction for $H_1$ in place of $K$, obtain $H_2$ and repeat the
construction again. We will eventually stop, namely obtain that
$H_i=H_{i+1}$, because every time when we repeat the construction if $H_i\neq H_{i+1}$ then there is some level $j$ such that on all the levels higher than $j$ the decompositions are the same as in the previous step and on level $j$ 
one of the following characteristics decreases:
\begin{enumerate}\item the number of free factors in the free
decomposition on  level $j$ of $H_i,$  \item if the number of free
factors does not decrease, then the number of edges and vertices of
the induced decomposition on  level $j$ of $H_i$ decreases,
\item if the number of free factors in the free decomposition on
 level $j$ of $H_i$  and the number of edges and vertices does not
decrease, then the size of the decomposition of level $j$ of $H_i$ is decreased.
\end{enumerate}
We need to recall here the following definition.
A  pair $(g, |\chi |)$ of  genus and  absolute value of the Euler characteristic of the surface corresponding to a QH subgroup $Q$ is called the {\em size} of $Q$ (and is denoted $size(Q)$).  A tuple $$size(D)= (size(Q_1),\ldots,size(Q_n))$$  of sizes of the MQH subgroups of the decomposition $D$ of a freely indecomposable group in decreasing order  is called the {\em regular size} of this decomposition.  We compare sizes left lexicographically. The abelian size  $ab(D)$ of $D$ is the sum of the ranks of abelian vertex groups
in $D$ minus the sum of the ranks of the edge groups for the edges
from them.  The size of $D$ is a pair $size(D), ab(D)$. 

 We end up with an NTQ system
$Q=1$ such that $K\leq F_{R(Q)}\leq F_{R(S)}$. and the image of the
top $j$ levels of $F_{R(Q)}$ on the level $j+1$ is contained in  the
image of $K$ on this level. Each QH subgroup of the induced system
is a finite index subgroup in some QH subgroup of $S=1.$ This NTQ system $Q=1$ is called the {\em induced NTQ system } and the corresponding well aligned  fundamental sequence is called the {\em induced fundamental sequence}.  The construction is algorithmic \cite{KMel}.

Similarly, if we have a fundamental sequence (and an NTQ system) modulo
a subgroup we  can define the induced fundamental sequence (and the NTQ
system)  modulo a subgroup.

\subsection{First step}\label{aex}
We will now describe the algorithm for construction of the $\exists\forall$-tree. Consider the sentence
\begin{equation}\label{ae}
\Phi=\forall  X\exists Y (U(X,Y)=1\wedge V(X,Y)\neq 1)\end{equation}

If the sentence is true then  there exists a solution of
$ U(X,Y)=1\wedge V(X,Y)\neq 1$ in $F(X)*F$ (\cite{Imp}, Theorem 4). By \cite{Mak} there is an algorithm to find such a solution $Y=f(X)$ or to say that there is no solution. If there is no such solution then  the sentence is false. Suppose there is a solution.  To check whether the sentence (\ref{ae}) is true we now have to check it only for  those values of $X$ for which  $V(X,f(X))=1.$  Denote $U_0(X)=V(X,f(X)).$ If $U_0(X)=1$ is inconsistent, then we say that the sentence is proved on the first level. 

 Let $G=\Gamma _{R(U_0)}$.
We define now a tree
$T_{EA}(\Phi)=T_{EA}(G)$ oriented from the root, and assign to each vertex of
$T_{EA}(G)$ some set of homomorphisms from $G$ to $\Gamma$. We assign the set of all homomorphisms  $G\rightarrow F$ to the initial vertex $w_{0}$.
 We can
construct algorithmically a finite number of NTQ systems  corresponding to branches $b$ of the canonical $Hom$-diagram described in Section \ref{can_diagram}  (denote the system corresponding to the branch $b$ by $S(b)$), their correcting extensions: $S_{corr}(b)=1$
($S_{corr}(b):S_1(X_1,\ldots ,X_n)=1,\ldots ,S_n(X_n)=1$), and corresponding  well aligned
fundamental sequences $Var_{\rm fund}(S(b))$. For each such fundamental
sequence we assign a vertex $w_{1,i}$ of the tree $T_{EA}(\Phi)$. We draw an edge from
 vertex $w_{0}$  to each vertex $w_{1,i}$ corresponding to $Var_{\rm fund}(S(b)).$

 Let the fundamental sequence 
$Var_{\rm fund}(S(b))$ be assigned to $w_{1,i}$. Since every branch of the tree will be constructed independently of the others we will now describe the construction for $Var_{\rm fund}(S(b))$. Denote $N_0=F_{R(S(b))}$. If the  image of $G$ in $N_0$    is a proper quotient of $G$ then we replace  $Var_{\rm fund}(S(b))$ by the disjunction of the fundamental sequences for this proper quotient. Suppose the image of $G$ is isomorphic to $G$. By  the
parametrization theorem (\cite{Imp}, Theorem 12)  Player A can find
values of $Y$ given by  formulas $Y=f(X_1,\ldots ,X_n)$ in 
$X_1,\ldots ,X_n$ taking values in  ${\Gamma}_{R(S(b))}$. Player B replies that sentence (\ref{ae}) is now verified for all values of $X$ except those for which 
  we have $V(X_1,\ldots ,X _n, f(X_1,\ldots ,X_n))=1.$  

We say that a fundamental sequence is constructed modulo some subgroups of the coordinate group  if these subgroups are elliptic in the JSJ decompositions on all levels in the construction of this fundamental sequence.
\begin{definition} \label{minhom} Let $S=1$ be a canonical NTQ system corresponding to a well-aligned fundamental sequence for a group $L$ relative to  limit groups  $L_1,\ldots ,L_k$:

\medskip
$S_1(X_1, X_2, \ldots, X_n,A) = 1,$

\medskip
$\ \ \ \ \ S_2(X_2, \ldots, X_n,A) = 1,$

$\ \ \ \ \ \ \ \ \ \  \ldots$

\medskip
$\ \ \ \ \ \ \ \ \ \ \ \ \ \ \ \ S_n(X_n,A) = 1$

 We say that values of $X_2,\ldots, X_n$ are {\em minimal with respect to the images of $L$} if for each $i>1$
values of $X_i,\ldots ,X_n$ are minimal in fundamental sequences for the groups $F_{R(S_i,\ldots, S_n)}$ modulo rigid subgroups and edge groups of the decomposition of the image of $L$ on level $i-1$ from the top. In particular, values of $X_2,\ldots, X_n$
are minimal in fundamental sequences for $F_{R(X_2,\ldots, X_n)}$ modulo rigid subgroups and edge groups of the image of $L$ in $F_{R(X_1,\ldots, X_n)}$.
\end{definition}

Player A replies that we will only consider values of $X_1,\ldots ,X_n$ satisfying this system that are minimal in $Var_{fund}(S(b))$ with respect to the images of $G$. This is enough because if for a specialization $X^{\phi}$ there exists a minimal specialization of $X_1^{\phi},\ldots ,X_n^{\phi}$ that does not satisfy this system, then the sentence is true for $X^{\phi}$. 

Let $Var_{\rm
fund}(U_1)$ be the subset of homomorphisms from the set $Var_{\rm
fund}(S(b))$ going through the corrective extension 
$S_{corr}(b)=1$, minimal with respect to the canonical automorphisms modulo the images of $G$ and satisfying the additional equation $V(X_1,\ldots ,X _n, f(X_1,\ldots ,X_n))=1.$   If this equation does not have any solution, then the sentence is proved on level 2.

Let $K$ be a finitely generated group. Recall that any family of homomorphisms $\Psi=\{\psi _i:K\rightarrow F \}$ factors through a finite set of maximal fully residually free groups $H_1,\ldots , H_k$ that all are quotients  of $K$.  We first take a quotient $K_1$ of $K$ by the intersection of the kernels of all homomorphisms from $\Psi$, and then construct maximal fully residually free quotients $H_1,\ldots , H_k$ of $K_1$. We say that $\Psi$ {\em discriminates} groups  $H_1,\ldots , H_k$, and that each $H_i$ is {\em a fully residually free group discriminated by $\Psi .$}

 Let $G_1$ be a fully residually free group discriminated by the set of homomorphisms
$Var_{\rm fund}(U_1)$. $G_1=F_{R(U_1)}$ for an irreducible system $U_1(X_1,\ldots ,X_n)=1.$ Consider the family  of  well-aligned fundamental
sequences for $G_1$  modulo the images $R_1,\ldots, R_s$ of the factors in the free  decomposition of the subgroup $H_1=\langle X_2,\ldots, X_n\rangle $.  We know the generators and relations of  $R_1,\ldots, R_s$ and  can effectively construct these fundamental sequences.
 Since we only consider well aligned canonical fundamental sequences $c$ for $G_1$ modulo $R_1,\ldots, R_s$  (corresponding to coefficients of quadratic
 equations $S_1=1$ of the top level for $S_{corr}(b)=1$)
they have, in particular, the following properties: (1) they have dimension less than or equal to
 $k_1$,  (2)  $c$ is consistent with the decompositions of surfaces corresponding to
 quadratic equations of $S_1$  by a collection of simple closed curves mapped to the identity by $\pi _1$. Namely, if we refine
 the JSJ decomposition of $G$ by adding splittings
corresponding to the simple closed curves that are mapped to the identity when $G$ is mapped to the free product in Subsection \ref{spl}, then the
standard coefficients on all the levels of $c$ are images of
elliptic elements in this decomposition.
 
 If $G$ is not embedded  in the block-NTQ group discriminated by the
fundamental sequence, then we  replace $G$ by its proper quotient that is embedded and continue this branch by replacing $G$  with this proper quotient.  

Now  we assume that $G$ is embedded in the NTQ group corresponding to the fundamental sequence $c$.

Suppose a fundamental sequence $c$ has the top dimension component
$k_1$. If the NTQ system corresponding to the top level of the sequence
$c$ is the same as $S_1=1$, we extend the fundamental sequences
modulo $R_1,\ldots ,R_s$ by canonical fundamental sequences  for
$H_1$  modulo  the factors in the free decomposition of the subgroup
$\langle X_3,\ldots ,X_n\rangle $. If such a sequence has dimension greater than
or equal to $k_2$, then the corresponding solution can be factored
through a fundamental sequence for $U=1$ of
 the greater dimension.
So we only consider such sequences of dimension less than or equal
to $k_2$. If the sum of first two dimensions is strictly smaller
than $k_1+k_2$, we do the same as in the case when the first dimension
is smaller than $k_1$ (see below). We continue this way to construct
fundamental sequences $Var_{\rm fund}(S_1(b))$. 

 Suppose now that the fundamental sequence $c$ for $G_1$ modulo
$R_1,\ldots ,R_s$ has dimension strictly less than $k_1$ or has
dimension $k_1$, but the NTQ system corresponding to the top level of
$c$ is not the same as $S_1=1$. Suppose also that $G\neq F$. Then we
use the following lemma (in which we suppose that $R_1,\ldots ,R_s$
are non-trivial).

\begin{lemma}\label{prop}The image $G_{t}$ of  $G$ in
the group $H_{t}$ appearing on the terminal level $t$ of the sequence $c$ is a
proper quotient of $G$ unless $G$ is a free group.
\end{lemma}
\begin{proof} Consider the terminal group of $c$; denote it $H_t$.
Suppose $G_{t}$ is isomorphic to $G$. Denote the abelian JSJ
decomposition of $H_t$ by $D_t$. Then there is an abelian
decomposition of $G$ induced by $D_t$. Therefore rigid (non-abelian
and non-QH) subgroups and edge groups of
$G$ are elliptic in this decomposition.

There exists a decomposition of some free factor $P_i$ of $H_t$ which is induced from
$D_t$. But this is impossible because this  means that
the homomorphisms we are considering can be shortened by applying
canonical
automorphisms of $P_i$ modulo those subgroups from $\{R_1,\ldots ,R_s\}$ which are conjugated into $P_i$ (since $V_{fund}(U_1)$  contains only homomorphisms minimal  in fundamental sequences modulo  $R_1,\ldots ,R_s$ (see Def. \ref{minhom}), these subgroups must be also elliptic in $D_t$).

\end{proof}

Therefore the image  of $G$ on all the levels of the fundamental sequence $c$ above some level $p$ is isomorphic to $G$ and on level $p$ is a proper quotient of $G$.  We can effectively find at which level this happens. Indeed, by \cite{KMel}, Theorem 29 we can find a presentation for the image of $G$ on each level and by \cite{KMel}, Theorem 31  we can decide if the epimorphism from $G$ to this image is proper. We can effectively find the set of canonical fundamental sequences and NTQ groups for $G_p$.   Denote the complete set  of fundamental sequences that encode all the homomorphisms  $G_p\rightarrow F$ by $\mathcal F$. One can extract from $c$ modulo 
level $p$
 the induced  well aligned fundamental sequence for $G$.  Denote this induced fundamental sequence up to level $p$ by $c_2$.   Consider a fundamental sequence $c_3$ that consists of homomorphisms obtained by the composition
  of a homomorphism from $c_2$ and from a fundamental sequence
$b_2\in \mathcal F.$ Let $\bar N_1$ be the  NTQ group corresponding to $c_3$.

Denote by $M(X,Z_1)$ (where $Z_1$ is the set of generators, $X\subset Z_1$) the group generated by  the top $p$ levels of the  NTQ group corresponding
 to the fundamental sequence $c$.  Consider the {\em block-NTQ group}  $N_1$ that is a fully residually free quotient 
 of the amalgamated product of $M(X,Z_1)$  and  the group $\bar N_{1corr}$ that is a  corrective extension of $\bar N_1$ corresponding to the system $U(X,Y)=1\wedge V(X,Y)\neq 1$,   amalgamated  along the top $p$ levels of $\bar N_{1corr}$. To obtain a fully residually free quotient we   add relations to make it commutative transitive.  Assign the  sequence $c_3$, block-NTQ group $N_1$, groups $\bar N_{1corr}$ and $M(X,Z_1)$ to the vertex $w_{2,i,k}$ of the tree $T_{EA}(\Phi)$.
 We draw an edge from the vertex $w_{1,i}$ of $T_{EA}(G)$  to $w_{2,i,k}$.   
  
 One can apply the parametrization theorem  (\cite{Imp}, Theorem 12) to  the NTQ group $\bar N_1$ and get a formula solution of $U(X,Y)=1\wedge V(X,Y)\neq 1$ in $\bar N_{1corr}$.     Those formula solutions for which $V(X,Y)=1$ 
 (if exist) will give an additional equation $U_2=1$ for generators of  $N_1$. Moreover,  we can suppose that $G_2=F_{R(U_2)}$ is discriminated by homomorphisms that are minimal for $\bar N_1$ with respect to images of $G$ and minimal  for 
$M(X,Z_1)$ with respect to images of $F_{R(U_1)}$.  (We call this assumption the {\em minimality restriction}). We assign groups $F_{R(U_2)}$ to vertices $\hat w_{2}$ (with different subscripts).

\subsection{Second step}\label{2nd}
We will describe the next  step in the construction of $T_{EA}(\Phi )$
which basically is general. Fundamental sequences and block-NTQ groups obtained on the second step will be assigned to vertices $w_{3}$ (with additional subscripts) of the tree $T_{EA}(\Phi)$.

The case when  the natural image  $G^{(2)}$ of  $G$ in $G_2$
 is a proper quotient of  $G$ is the first ``easy'' case.   By \cite{KMel}, Theorem 31, there is an algorithm to determine if this is the case.
 If $G^{(2)}$ is a proper quotient of $G$, we continue the construction along this branch as on the first step but with $G$ replaced by $G^{(2)}$.
 In this case we assign
 to vertices $w_{3}$ the fundamental sequences constructed as in Proposition \ref{pr:solutions-special-Sb}  for $G^{(2)}$ (corresponding to the branches of $T_{CE}(G^{(2)})$) and draw edges from $ w_{2}$ to these vertices.   In all the other considerations below we suppose that $G^{(2)}$ is
isomorphic to  $G$.

Suppose  the JSJ decomposition for the NTQ system corresponding to the top level of $c$ corresponds to
the equation $S_{11}(X_{11},X_{12},\ldots )=1;$ some of the
variables $X_{11}$ are quadratic, the others correspond to
extensions of centralizers.
 Construct a canonical fundamental sequence $c^{(2)}$  for $G_2$ modulo  the factors in the free decomposition of the subgroup generated by
$X_{12},\ldots  $.

Denote by $N_0^1$ the image of the subgroup generated by $X_1,\ldots
,X_n$ in the group $M(X,Z_1)$ discriminated by $c$. So, $N_0^1= \langle X_1,\ldots
,X_n\rangle _c.$ Denote by $N_0^2=\langle X_1,\ldots ,X_n\rangle _{c^{(2)}}$ the image of
$\langle X_1,\ldots ,X_n\rangle $
 in the group discriminated by $c^{(2)}$. If  $N_0^2$ is a proper quotient of $N_0^1$ 
 and $c$ is not the same as the top level of $S(b)$, we have another ``easy'' case.
 In this case we do what we did on the previous step, taking $N_0^2$
instead of $G_2$, assigning its well aligned fundamental sequences to $\hat v_{2ijk}$, and we do not need to consider vertices corresponding to NTQ systems with the same top level as the NTQ system for $c_3$ because the complexity (see the definition in Section \ref{5.6}) of $c^{(2)}$   was already less than the complexity of $c_3$. 
 
In all the cases below we suppose that $N_0^1$ is isomorphic to
$N_0^2$.

{\bf Case 1.} If the top levels of $c$ and $c^{(2)}$ are the same, then
we go to the second level of $c$ and consider it the same way as the
first level.

{\bf Case 2.}  If the top levels of the NTQ system for $c$ and $S_1$ are the same
(therefore $c$ has only one level). We work with $c^{(2)}$ the same
way as we did for $c$. Suppose $c^{(2)}$ is not the same as $c$.   Then  the image of $G$ on some level
$p$ of $c^{(2)}$ is a proper quotient of $G$ by
Lemma \ref{prop}. Let $M(X,Z_2)$ be the NTQ group corresponding to the top $p$ levels of $c^{(2)}$. We consider fundamental sequences constructed as follows: the top part is the fundamental sequence induced by the top part of $c^{(2)}$ 
above level $p$ for $G$, and the bottom part is a fundamental sequence for this quotient of $G$ (solutions will go along the first fundamental sequence from the top level to level $p$ and then continue along one of the fundamental sequences for the image of $G$ on level $k$).  We assign each of these fundamental sequences to a vertex  $w_{3,j,k,s}$, then take the corresponding canonical NTQ group $\bar N_2$. Then we construct the block-NTQ group as we did on the first step, denote it by $N_2$. We also assign $N_2$ and $\bar N_{2corr}$ to the vertex  $w_{3,j,k,s}$.

{\bf Case 3.}  If the top levels of the NTQ system for $c$ and $S_1$ are  not  the same and
the top levels of $c$ and $c^{(2)}$ are not the same, then we look at  $N_0^2$ and $N_0^1$.  

If   $N_0^2$ is isomorphic to $N_0^1$ then it follows from the minimality restriction made for the solutions used to construct $F_{R(U_2)}$ that there is some
 level $k$ from the top of $c^{(2)}$ such that we can suppose that the image of
either $G$ or $N_0^1$ on this level is a proper quotient of it (and on the levels above k is isomorphic to $G$ (resp. to $N_0^1$). 

If the image of $G$ that we denote $G_t$ is a proper quotient of $G$ we do what we did on step 1. Namely, we construct the fundamental sequence for  $G$ induced from $c^{(2)}$ and continue it with a canonical fundamental sequence for $G$, construct NTQ group $\bar N_2$ for this fundamental sequence and the block NTQ group $N_2$ and assign them to a vertex $w_3$.
 
Suppose now that $G_t$ is isomorphic to $G$ and $N_{0,t}^1$ is a proper quotient of $N_0^1$. Consider
fundamental sequences for $N_{0,t}^1$ modulo the images of subgroups $R_1,\ldots, R_s$, and apply to them step 1. Denote the obtained fundamental sequences  by $f_i$.
Construct fundamental sequences for the subgroup generated by the images of 
$X_1,\ldots ,X_n$ with the top part being induced from the top
part of $c^{(2)}$ (above level k) and bottom part being some $f_i$,
but not the sequence with the same top part as $c$. We construct the block-NTQ group amalgamating the top $k$ levels of $c^{(2)}$ and the block-NTQ group constructed for $f_i$ as on the first step. We denote this block-NTQ group (consisting of three blocks) by $N_2$ and its subgroup that is a canonical NTQ group by $\bar N_2$. These groups  are assigned to $\hat v_2.$ 

If the sentence is true there exists a formula solution over the correcting extension of the group $\bar N_2$.

{\bf Case 4.} If all the levels of $c$ and $c^{(2)}$ are the same, so we never have cases 2 and 3,  and the same happens on all levels of the block-NTQ group $N_1$ (see step 1), then 
the block fundamental sequence  consists of a sequence of induced fundamental sequences
for $G$ and its images. Denote it $\overline c_2$. For each level of $\overline c_2$ there is an abelian decomposition. 
 Denote by $G_{corr}$ the image of the corrective extension of the NTQ group corresponding to the fundamental sequence induced by $G$ and its images with which we came to this step (that was denoted by $\bar N_{1corr}$).  Denote  the fundamental sequence induced from $c$ and its continuation by $G_{corr}$  by $c_4$.   The group $G_{corr}$ is not embedded into the block-NTQ group $N_1$.
Therefore there exists some level $k$ such that the abelian decompositions for the NTQ group  for $c_4$ will coincide with abelian decompositions for  the NTQ group for $\overline c_2$  for levels above $k$, and on level $k$ either the number of free factors in the free decomposition  for $c_4$ is less than this number for   $\overline c_2$, or the number of factors is the same, but the  regular size of the decompositions (lexicographically ordered tuple $(size (Q_1),\ldots, size (Q_m))$  of sizes of  MQH subgroups)  for $c_4$ is  smaller  than that  for $\overline c_2$, or the regular sizes are the same but the abelian size $ab$ of the decompositions  for $c_4$ is  smaller  than that for $\overline c_2$.  Here, if $R$ is an abelian decomposition,
by $ab(R)$ (the abelian size) we denote the sum of the ranks of abelian vertex groups
in $R$ minus the sum of the ranks of the edge groups for the edges
from them. 

We will take $c_4$ to the next step instead of $\overline c_2$.

\subsection{General step}\label{nth}
We now describe  the $n$'th step of the construction. Denote by
$N_i$ the block-NTQ group constructed on the $i$'th step (and assigned to some vertex $w_{i+1}$ with subscripts), and by
$N_i^j, j> i$ its image  on the $j$'th step. Denote the image of $G$ on step $i$ by $G^{(i)}.$
Fundamental sequences and block-NTQ groups obtained on  step $n-1$ are assigned to vertices $w_{n,j,k,s\ldots }$ of the tree $T_{EA}(G)$. 

Let $\{j_k,\
k=1,\ldots ,s\}$ be all the indices for which the top level of
$N_{j_k+1}$ is different from the top level of $N_{j_k}$. Let $M(X, Z_i)$ be the group corresponding to the top block on step $i$ and $M^j(X, Z_i)$ its image in $N_{j-1}^j.$

The first ``easy" case is when  $G$  (or its proper quotient if we are already working with a proper quotient of $G$ on step $n-1$) is not embedded
into $N_{n-1}^n$, then we assign to vertices $w_{n+1}$ (with subscripts)  well aligned fundamental sequences for $G^{(n)}$ and  their canonical NTQ groups.

On each step $i$ we consider fundamental sequences for the proper quotient of $N_{i-1}$
 modulo the images of freely indecomposable free factors  of the second level block NTQ group $N_{i-1}$ (images of rigid subgroups in the relative decomposition of $M(X, Z_{i-1})$). Let $R^{r(i)}$ be the family of these rigid subgroups on step $i$. Every time when we increase  parameter subgroups $R^{r(i)}$, $r(i)$ increases by 1.  Let $g(t)$ be the last step  when  we constructed fundamental sequences modulo $R^t$. On step $g(t)+1$ parameters subgroups are increased to become $R^{t+1}$ and fundamental sequences for the top block are constructed modulo $R^{t+1}$, so $r(g(t)+1)=t+1.$

The second ``easy" case  appears when $G$ is isomorphic to $G^{(n)}$ but for some $g(r)<n,$ $M^n(X, Z_{g(r)})$ is a proper quotient of $M^{g(r)}(X, Z_{g(r)})$.  Let $r$ be the smallest such index. Then we assign this proper quotient to $w_{g(r)+1}$ and do not continue along the branch originating from $w_{g(r)+1}$ and arriving at $w_n$. Instead, we generate another branch originating from $w_{g(r)+1}$  and work along this branch  with $M^n(X, Z_{g(r)})$ assigned to $w_{g(r)+1}$. Working   along this branch we only consider fundamental sequences for $M^{g(r)}(X, Z_{g(r)})$ modulo increased parameter subgroups because they were increased previously on step $g(r)+1$, therefore
on step $g(r)+1$ solutions were already not of maximal complexity.
 In all other cases
we can suppose that  $G$ and all the groups $M^{g(r)}(X, Z_{g(r)})$, $g(r)<n,$ are
embedded into $N_{n-1}^n.$

{\bf Case 1.} The top levels of $c^{(n)}$ and $c^{(n-1)}$ are the
same. In this case we go to the second level and consider it the
same way as the first level.

{\bf Case 2.} The top levels of $c^{(n-1)}, c^{(n-2)},\ldots ,
c^{(n-i)}$ are the same, and the top levels of $c^{(n-1)}$ and $c^{(n)}$
are not the same. Then on some level $p$ of the NTQ group for $c^{(n)}$ we can suppose
that the image of $M^n(X,Z_{n-1})$ is a proper quotient of it (and on the levels above $p$ it is  isomorphic to $M^n(X,Z_{n-1})$). Let $r$ be the minimal such index that  $M^n(X,Z_{g(r)})_t$ is a proper quotient of $M^n(X,Z_{g(r)})=M^{g(r)}(X,Z_{g(r)})$. Then the levels below the first level (level 2 and below) of the block NTQ group $N_n$ that we are constructing on step $n$  will correspond to the block NTQ group $\tilde N_{g(r)+1}$ that we would construct on step $g(t)+1$ for $M^n(X,Z_{g(r)})_t$.  Moreover, we consider only fundamental sequences for $\tilde N_{g(r)+1}$ with the top level different from $c^{(g(r))}$ (not of maximal complexity). 

{\bf Case 3.} The top levels of $c^{(n-2)}$ and $c^{(n-1)}$ are not
the same and the top levels of $c^{(n-1)}$ and $c^{(n)}$ are not the
same. Then on some level $p$ of $c^{(n)}$ the image of
$M^n(X,Z_{n-1})$  is a proper quotient of $M^n(X,Z_{n-1}).$  Let $r$ be the minimal such index that  $M^n(X,Z_{g(r)})_t$ is a proper quotient of $M^n(X,Z_{g(r)})=M^{g(r)}(X,Z_{g(r)})$. Then the levels below the first level (level 2 and below) of the block NTQ group $N_n$ that we are constructing on step $n$  will correspond to the block NTQ group $\tilde N_{g(r)+1}$ that we would construct on step $g(t)+1$ for $M^n(X,Z_{g(r)})_t$.  Moreover, we consider only fundamental sequences for $\tilde N_{g(r)+1}$ with the top level different from $c^{(g(r))}$ (not of maximal complexity).

{\bf Case 4.} 
If going from the top to the bottom of the block-NTQ system, on some level of the block-NTQ group we can apply Case 2 or 3 to this level we apply it.

If going from the top to the bottom of the block-NTQ system we do not obtain  Cases 2, 3 and all the levels  of $N_{n-1}$ and $N_n$ are the same, we  
consider the group $G_{n, corr}$ which is the image of the group that was constructed for the induced
fundamental sequence corresponding to the homomorphisms from $G$
going through $N_{n-1}$. Let $G_n$ be the NTQ group for that induced fundamental sequence, so that $G_{n,corr}$ is the image of a corrective extension for $G_n$.  We construct the fundamental sequence $d$ induced by $G_{n,corr}$  from $N_{n}$ (from $c^{(n)}$ and its continuation) as we did on the second step when cases 2 and 3 were not applicable on all the levels of $c^{(2)}$.   There exists some level $k$ such that the abelian decompositions for the NTQ group  for $d$ will coincide with abelian decompositions for $G_{n}$  for levels above $k$, and on level $k$ either the number of free factors in the free decomposition  for the NTQ group for $d$ is less than this number for  $G_{n}$ , or the number of factors is the same, but the  regular size of the decomposition is  smaller  than that  for $G_{n}$, or the regular sizes are the same but the abelian size $ab$  is  smaller  than that for $G_{n}$. The NTQ group corresponding to  d will be denoted $\bar N_n$. We will go to the next step with
$N_n$, $\bar N_{ncorr}$ and the fundamental sequence $d$.

On  each step of  the construction of the tree $T_{EA}(\Phi)$  if we can apply the first easy case we apply it. If we cannot apply the first easy case but can apply the second easy case we apply it. Otherwise we apply one of the other cases 

\subsection{The $\exists\forall$ tree is finite}\label{5.6}

It is convenient to define  the notion of complexity of a fundamental sequence ($Cmplx (Var_{fund})$) at follows:
$$Cmplx(Var_{\rm fund}) =$$ $$(dim (Var_{\rm fund}) + factors (Var_{\rm fund}), (size (Q_1),\ldots ,size (Q_m)), ab(Var _{\rm fund}(Q)),$$
where $factors (Var_{\rm fund})$ is the number of freely indecomposable, non-cyclic terminal factors (this component only appears in fundamental sequences relative to  subgroups),  and   $(size (Q_1),\ldots, size (Q_m))$ is the regular size of the system. The complexity is a tuple of numbers which we compare in the left lexicographic order. (Sela calls
 $dim (Var_{\rm fund}) + factors (Var_{\rm fund})$ the Kurosh rank of the resolution.)

In this subsection we will prove the following result.

\begin{theorem} \label{AE} The tree $T_{EA}(\Phi)$ and the algorithm for its construction  are finite.\end{theorem}
Proof.
 Let an NTQ system $Q(X_1,\ldots, X_n)=1$ have the form
$$Q_1(X_1,\ldots ,X_n)=1,$$
$$\ldots $$
$$Q_n(X_n)=1,$$
where $Q_1=1$ corresponds to  the top level of a JSJ decomposition for
$\Gamma _{R(Q)}$, variables from $X_1$ are either quadratic, or correspond
to extensions of centralizers. Consider this system together with the 
fundamental sequence $V_{\rm fund}(Q)$ defining it. Let $V_{\rm
fund}(U_1)$ be the subset of $V_{\rm fund}(Q)$ satisfying some
additional equation $U_1=1$, and $G_1$ a group discriminated by this
subset. Consider the family  of those canonical fundamental
sequences for $G_1$ modulo the images $R_1,\ldots ,R_s$ of the
factors $P_1,\ldots ,P_s$ in the free decomposition of
 the subgroup $\langle X_2,\ldots ,X_m\rangle $, which have the same Kurosh rank modulo them as
 $Q_1=1$. Denote this free decomposition by $H_1*$. The terminal group of such a fundamental sequence does not have a sufficient splitting relative to  $H_1*$.

Denote such a fundamental sequence by $c$, and the corresponding NTQ
system $S=1 (mod \ H_1*)$, where $S=1$ has the form
$$S_1(X_{11},\ldots ,X_{1m})=1$$
$$\ldots $$
$$S_m(X_{1m})=1.$$

Denote by $D_Q$ a canonical decomposition corresponding to the group
$F _{R(Q)}$. Non-QH, non-abelian subgroups in this decomposition are
$P_1,\ldots ,P_s$. Abelian and QH subgroups correspond to the system
$Q_1(X_1,\ldots ,X_n)=1.$ For each $i$ there exists a canonical
homomorphism
$$\eta _i: F _{R(Q)}\rightarrow F _{R(S_i,\ldots ,S_m)}$$
such that $P_1,\ldots ,P_s$ are mapped into
rigid subgroups in the canonical decomposition of $\eta
_i(F _{R(Q)})$.

Each QH subgroup in the decomposition  of $F _{R(S_i,\ldots ,S_m)}$
as an NTQ group is a QH subgroup of $\eta _i(F _{R(Q)})$. By \cite{KMel}, Lemma 11,
for each QH subgroup $Q_1$ of $\eta
_i(F _{R(Q)})$ there exists a QH subgroup of $F _{R(Q)}$ that is
mapped into a subgroup of finite index in $Q_1$. The size of this QH
subgroup is, obviously, greater or equal to the size of $Q_1$. Those
QH subgroups of $F _{R(Q)}$ that are mapped into QH subgroups of the
same size by some $\eta _i$ are called stable.

\begin{lemma}\label{ch}
In the conditions above there are the following possibilities:

(i) The set of homomorphisms going through $c$ is generic for each
regular quadratic equation in $Q_1=1$ and $ab(c)=ab(V_{\rm
fund}(Q))$ (in this case $c$ has only one level identical to $Q_1$);

(ii) It is possible to reconstruct system $S=1$ in such a way that $size\ (S) < 
size\ (Q_1)$;

(iii)  $size\ (S) = size\ (Q_1)$, $ab(c)< ab(V_{\rm fund}(Q)).$

\end{lemma}

\begin{proof}

 The fundamental sequence $c$ modulo the decomposition $H_1* $
has the same Kurosh rank as $Q_1=1$. The Kurosh rank of $Q_1=1$ is the
sum of the following  numbers: \begin{enumerate} \item [1)] the
dimension of a free factor $F_1=F(t_0,\ldots ,t_{k_0})$ in the free
decomposition of $F_{R(Q)}$ corresponding to an empty equation in
$Q_1=1$; \item [2)] the number of abelian factors;
\item [3)] the sum of dimensions of surface group factors; \item
[4)] the number of free variables of quadratic equations with
coefficients in $Q_1=1$ corresponding to the fundamental sequence
$Var _{\rm fund}(Q)$, \item [5)] $factors (Var_{\rm fund})$.\end{enumerate} Because $c$ has the same
Kurosh rank, the free factor $F_1$ is unchanged. By 1) and 2)  in
Section \ref{rk3}, abelian and surface factors are sent into
different free factors.

Let $Q_{1i}=1$ be one of the standard quadratic equations in the
system $Q_1=1.$ If the set of solutions of $Q_{1i}=1$ over $F_{R(
Q_2,\ldots Q_n)}$ that factor through the system $S=1$ is a generic family for
$Q_{1i}=1$, then by the analog of (\cite{Imp},  Theorem 9)  we conclude that $S=1$
can be reconstructed so that it contains only one quadratic equation
as a part of the system $S_m=1$. Indeed, suppose a QH subgroup
$\overline Q_{1i}$  corresponding to $Q_{1i}=1$ mapped on some level
$s$ of $S=1$ onto a subgroup of the same size. Then it is stable.
Suppose also that a QH subgroup of $F_{R(Q)}$ that is a subgroup of
$\overline Q_{1i}$ is projected on some level $k$ above $s$ into a
QH subgroup $\overline Q_k$. Then this projection is a monomorphism.
On all the levels above $s$ we can adjoin the image of a subgroup of
$\overline Q_{1i}$ to a non-QH subgroup adjacent to it (and not
count it in the size). We can adjoin the image of $\overline Q_{1i}$
to a non-QH subgroup on all the levels above $m$, and replace the
image of it on level $m$ by the isomorphic copy of $\overline
Q_{1i}$.

If all QH subgroups corresponding to  $Q_1=1$ are stable, then the
regular size of $S=1$ is the same as the regular size of $Q_1=1$ and
if $ab(c)=ab(V_{\rm fund}(Q))$, then reconstructed $S=1$ has only
one level.

The lemma  is proved.
\end{proof}

We now will finish the proof of  Theorem \ref{AE}. Recall that proper descending chains of fully residually free quotients of a finitely generated fully residually free group are finite. We will use induction on the length of  proper descending chain of fully residually free quotients of $G$ and on the complexity of $S_1(X_1,\ldots , X_n)=1$.  The induction assumption is that for any fully residually free group with  the length of  proper descending chain of fully residually free quotients less than for $G$  and for any block-NTQ system  for which the complexity of the first level is less than the complexity of $S_1(X_1,\ldots , X_n)=1$ the procedure is finite. Suppose the procedure for the construction of the $\exists\forall$-tree for $G$ is infinite. 

If we obtained the first ``easy" case along some branch,  namely the case when for some $n$ the image $G^{(n)}$ of the group $G$ is a proper quotient of $G$, then the subtree    rooted at this vertex is finite by induction. 

The second easy case appears when $G$ is isomorphic to all $G^{(n)}$ but for some $g(r)<n,$ $M^n(X, Z_{g(r)})$ is a proper quotient of $M^{g(r)}(X, Z_{g(r)})$.  Let $r$ be the smallest such index.  In this case we stop this branch and originate another branch from the vertex corresponding to step  $g(r)$  replacing the group $M^{g(r)}(X, Z_{g(r)})$ with its proper quotient $M^n(X, Z_{g(r)})$ . Moreover, since on step $g(r)+1$ the parameters were increased  and, therefore, the complexity of the fundamental sequences for $M^{g(r)+1}(X, Z_{g(r)})$ was not maximal and all fundamental sequences are well aligned,  we only consider along this new branch  fundamental sequences for $M^n(X, Z_{g(r)})$ of non-maximal complexity. Therefore the number $g(r)$ does not change. For each number $r$ 
every time we go back to step $g(r)$ we add a term to a descending chain of proper quotients of $M^{g(r)}(X, Z_{g(r)})$.  Therefore for each $r$ this step can be repeated only a finite number of times.

By Lemma \ref{ch},
every time we apply the transformation of Case 3 (we refer to the
cases from Section \ref{nth}) in the construction of $AE$-tree we
either (i) decrease the dimension in the top block, therefore decrease the Kurosh rank, or (ii) replace
the NTQ system in the top block  by another NTQ system of the same
dimension but of a smaller size, or (iii) decrease $ab(c)$. Hence the complexity defined in the beginning of this section decreases. Notice that the complexity of the top block is bounded by the complexity of $S_1(X_1,\ldots ,S_n)=1.$ Hence,
Case 3 cannot be applied infinitely many times to the top block. 

Therefore starting at some step $n_1$ for any $n>n_1$, $r(n)=r(n_1)$, the top blocks of all $N_n$ are the same and fundamental sequences modulo $R^{r(n)}=R^{r(n_1)}$  for the top block are of maximal complexity. 

Case 4  together with $N_n=N_{n_1}$ cannot appear infinitely many times for $n>n_1$. Therefore eventualiy beginning at some step after step $n_1$ we are  applying the procedure to the second block, more precisely to the terminal level of first block fundamental sequence.
If at some step $n>n_1$ $G^{(n)}_t$ is a proper quotient of $G$ then  the procedure for the second level is eventually applied to $G^{(n)}_t$ and  is finite by induction.  Otherwise 
if
we apply Case 2, we consider the second block for  proper quotients
of a finite number of groups.  Let $r<r(n_1)$ be the minimal index for which $N^n(X, Z_{g(r)})_t$ is a proper quotient of $N^{g(r)}(X, Z_{g(r)})$ for some $n>n_1.$  Moreover we only consider fundamental sequences for $N^n(X, Z_{g(r)})_t$ modulo $R^{r+1}$  that are not of maximal complexity, in particular, their complexity is less than the complexity of 
$S_1(X_1,\ldots , X_n)=1$. By induction, the procedure for $N^n(X, Z_{g(r)})_t$ is finite.

  Theorem \ref{AE} is proved.

The effectiveness of the  construction of the finite $\exists\forall$-tree for  sentence
$\Phi$ implies that the $\forall\exists$-theory of the group $F$ is decidable.
\begin{theorem}\cite{KMel} The $\forall\exists$-theory of a free group  is decidable.
\end{theorem}

\newpage
\section{Effectiveness of the global bound in finiteness results}
\subsection{Groups without sufficient splitting, algebraic, reducing and special solutions}
 In Section 5.4 of \cite{KMel} we defined the notion of a sufficient splitting of a limit group $K$ modulo a class of subgroups ${\mathcal K}$.
Let $F$ be a free group with basis $A$, $P=A\cup \{p_1,\ldots ,p_k\}$, $H=\langle P\rangle .$ Let ${\mathcal K}$ consist of one subgroup ${\mathcal K}=\{H\}.$ Let  $K=\langle X, P, A|S(X,P,A)\rangle $,  and suppose that $K$ does not have a sufficient splitting modulo $H$.    Let $D$ be an abelian JSJ decomposition of $K$ modulo $H$.

We give the notion of algebraic solutions. Let $K_1$  be a fully residually free quotient of the group $K$, $\kappa: K \rightarrow
K_1$ the canonical $\Gamma$- epimorphism that embeds terminal subgroups of $\Gamma$ participating in the construction of $K$, and $H_1 = H^\kappa$ the canonical
image of $H$ in $K_1$.
An elementary abelian splitting of $K_1$ modulo $H_1$ which does not
lift into  $K$ is called a {\em new} splitting.
\begin{definition}
\label{de:reducing} (Definition 20 \cite{KMel}) In the notation above the quotient $K_1$ is
called {\em reducing} if one of the following holds:
\begin{enumerate}
\item $K_1$ has  a non-trivial free decomposition modulo $H_1$;
\item $K_1$ has a new elementary abelian splitting modulo $H_1$.

\end{enumerate}
\end{definition}

 We
say that a homomorphism $\phi :K\rightarrow K_1$ is {\em special}
if $\phi $ either maps an edge group of $D$ to the identity  or
maps
 a non-abelian vertex group of $D$ to an abelian subgroup.

We will now define $\sim_{MAX}$-equivalent homomorphisms (that were introduced  in \cite{KMel}, Section 5.3). Let $S$ be an elementary abelian splitting of a fully residually free group $G$  relative to a family of subgroups ${\mathcal
K}$, i.e., $G=A*_{C}B$  or $G=A*_{C}=\langle A,t|c^{t}=c', c\in
C\rangle$ modulo $\mathcal K$. Suppose, for certainty, that $F \leq
A$.
 Let $\psi:G\rightarrow F$  be an $F$-homomorphism from $G$
into $F$ and  $C^{\psi} \leq \langle c_0\rangle$, where $\langle
c_0\rangle$  is a maximal abelian subgroup of  $F$. For an arbitrary
$d \in \langle c_0\rangle$ we define a homomorphism $\psi_{d}
:G\rightarrow $ as follows. If $G=A*_{C}B$ then
 $$\psi_{d}(a)=\psi (a) \ for \ a \in A, \ \ \ \psi_{d}(b)=\psi (b)^{d} \ for \  b\in B.$$
   If
$G =\langle A,t|c^{t}=c', c\in C\rangle$ then
 $$\psi_{d} (a)=\psi (a) \  for  \ a\in A, \ \ \ \psi_{d}(t)=d \psi(t).$$
By $\sim_S$ we denote the following binary relation on $Hom_F(G,F)$
(in the notation above)
 $$ \sim_S = \{(\psi,\psi_d) \mid \psi \in Hom_F(G,F), d \in \langle
 c_0\rangle \}.$$

Now let $D$ be an   abelian JSJ decomposition of $G$ modulo
${\mathcal K}$. Suppose  $M$ is an abelian vertex group in $D$. Then
$M$ is a direct product $M=M_1\times M_2,$ where $M_1$ is the
minimal direct summand of $M$ containing all the edge groups of $M$
in $D$ (so the subgroup generated by the edge groups of $M$ has a
finite index in $M_1$). Denote by $G'$ the subgroup of $G$ which is
the fundamental group of the splitting $D'$ obtained from $D$ by
removing the direct summand $M_2$ from the vertex $M$. Clearly, $G$
splits as an extension of centralizer $C_{G'}(M_1)$ of the group
$G'$ by $M_2$. We fix a basis $g_1, \ldots, g_s$ of the free abelian
group $M_2$ (if $M_2 \neq 1$).

Now let $\theta : G \rightarrow F$ be an $F$-homomorphism and
$M^{\theta} \leq \langle c_0\rangle$, where $\langle c_0\rangle$ is
a maximal cyclic subgroup of  $F$. Then for every tuple $d = (d_1,
\ldots, d_s) \in \langle c_0\rangle^s$ the map
  $$\theta_d : g_i \rightarrow d_ig_i^\theta, \ \ i = 1, \ldots,
  s$$
 extends to a homomorphism $\theta_d : M_2 \rightarrow F$.
 Now the restriction of the homomorphism $\theta$ on $G'$ and the
 homomorphism $\theta_d : M_2 \rightarrow F$ give rise to a
 homomorphism $G \rightarrow F$ which we define by the same symbol
 $\theta_d$. We refer  to the homomorphism $\psi_{d}$ and $\theta_d$
as obtained from $\psi$ and $\theta$ by {\em extended
automorphisms} or {\em fractional Dehn twists}.

By $\sim_M$ we denote the following binary relation on $Hom_F(G,F)$
(in the notation above)
 $$ \sim_M = \{(\theta,\theta_d) \mid \theta \in Hom_F(G,F), d \in \langle
 c_0\rangle^s \}.$$

We extend the relation $\sim$ of being equivalent with respect to the group of canonical automorphisms to the equivalence relation $\sim
_{AE}$ generated by $\sim$, all the binary relations $\sim_M$ where
$M$ runs over all abelian vertex groups in $D$, and all the binary
relations $\sim_S$ where $S$ runs over all elementary splittings of
$G$ corresponding to the edges of $D$.

We say that two $F$-homomorphisms $\phi, \psi \in Hom_F(G,F)$
are ${MAX}$-equivalent (and write $\phi \sim_{MAX} \psi$) if
 there exists $\theta \in Hom_F(G,F)$ such that
  $\phi \sim _{AE} \theta$ and $\theta$  coincides up to conjugation with $\psi$
   on the fundamental group of
  every connected component of the graph of groups obtained from
  $D$ by   removing from $D$ all QH-subgroups. 

Let ${\mathcal R}=\{K/R(r_1), \ldots, K/R(r_s)\}$ be a complete
reducing system for $K$ (each homomorphism from $K$ into $F$ that factors through a reducing quotient is $\sim_{MAX}$-equivalent  to a homomorphism that factors through one of them).  The existence of such system  for a free group is proved in \cite{KMel}. The algorithm to construct it was also given in \cite{KMel}.  Suppose we know NTQ groups for the system of reducing quotients of $K$ modulo $H$.  A homomorphism from $K$ onto $F$ is called   {\em
reducing}
  if there exists a solution   the $\sim_{MAX}$-equivalence class
  of $\psi$ which  factors through  one of the NTQ systems for equations $r_1= 1, \ldots, r_k = 1.$
Now we define algebraic  solutions of $S = 1$ in $F$.
  Let $\phi :H \rightarrow F$ be a
fixed $F$-homomorphism and $Sol_\phi$ the set of all homomorphisms
from $K$ onto $F$ which extend $\phi$.
A non-reducing non-special solution in $Sol_\phi$ is called
{\em $K$-algebraic} (modulo $H$ and $\phi$).


\begin{theorem}  \label{th5} Let $H\leq K$ be as above. The fact that for parameters $P$ there are exactly $N$ non-equivalent Max-classes of $K$-algebraic solutions
of the equation $S(X,P)=1$ modulo $H$ can be written algorithmically as a boolean combination of conjunctive $\exists\forall$-formulas (these are formulas of type
(\ref{AE})).
\end{theorem}
\begin{proof} The generating set
$X$ of $K$ corresponding to the decomposition $D$ can be partitioned
 as $X=X_1\cup X_2$ such that $G=\langle X_2\cup P\rangle $ is the  fundamental group of the graph of groups obtained from $D$ by removing all QH-subgroups.
If $c_e$ is a given generator of an edge group of $D$, then we know how a generalized fractional Dehn twist (AE-transformation or extended automorphism in the terminology of \cite{KMel}, \cite{Imp}) $\sigma$ associated with edge $e$ acts on the generators from the set $X$. Namely, if $x\in X$ is a generator of a vertex group, then either $x^{\sigma}=x$ or $x^{\sigma}=c^{-m}xc^{m},$ where $c$ is a root of the image of $c_e$ in $F$, or in case $e$ is an edge between abelian and rigid vertex groups and $x$ belongs to the abelian vertex group, $x^{\sigma}=xc^{m}$. Similarly, if $x$ is a stable letter then either $x^{\sigma}=x$ or $x^{\sigma}=xc^{m}.$

 One can write elements $c_e$ as words in generators $X_2$, $c_e=c_e(X_2).$ Denote $T=\{t_i,\ i=1,\ldots ,m\}.$ Consider the  formula
\begin{multline*}
\exists X_1 \exists X_2\forall Y \forall T\forall Z \left(
S(X_1,X_2,P)=1\right.\\  \wedge \neg \left(\left.\bigwedge_{i=
1}^{m}[t_i,c_i(X_2)]=1 \wedge Z=X_2^{\sigma_T}\wedge
S(Y,X_2,P)=1\wedge V(Y,Z,P)=1\right)\right).
\end{multline*}
It says that there exists a solution of the equation $S(X_1,X_2,P)=1$ that is not Max-equivalent to a solution $Y,Z,P$ that
satisfies $V(Y,Z,P)=1$.  If now $V(Y,Z,P)=1$ is a disjunction of equations defining maximal reducing quotients, then this formula states that for
parameters $P$ there exists at least one Max-class of algebraic solutions of $S(X,P)=1$ with respect to $H$.

Denote $$\tau(T,X_2,Y,Z)=\left(\bigwedge_{i=
1}^{m}[t_i,c_i(X_2)]=1 \wedge Z=X_2^{\sigma_T}\wedge
S(Y,X_2,P)=1\wedge V(Y,Z,P)=1\right).$$The following formula states  that for
parameters $P$ there exists at least two non-equivalent Max-classes of algebraic solutions of $S(X,P)=1$ with respect to $H$.
\begin{multline*}
\theta _2(P)=\exists X_1, X_3 \exists X_2, X_4\forall Y, Y' \forall T, T',T''\forall Z,Z' \left(
S(X_1,X_2,P)=1\wedge S(X_3,X_4,P)=1\right.\\  \wedge \neg \left(\left.\tau(T,X_2,Y,Z)\vee \tau(T',X_4,Y',Z')\vee (\bigwedge_{i
=1}^{m}[{t_i}'',c_i(X_2)]=1\wedge X_2^{\sigma_{T''}}=X_4) \right)\right).
\end{multline*}

 Similarly one can write a formula $\theta _N(P)$ that states for
parameters $P$ there exist at least N non-equivalent Max-classes of algebraic solutions of $S(X,P)=1$ with respect to $H$.

Then $\theta _N(P)\wedge \neg\theta _{N+1}(P)$ states that there are exactly $N$ non-equivalent Max-classes. The theorem is proved.
\end{proof}
\subsection{A bound on the number of  Max-classes of algebraic solutions}
 In this subsection we will give a proof of the effectiveness of the global bound in Theorem 11 \cite{KMel}.
\begin{theorem}\label{term} (\cite{KMel}, Theorem 11) Let $H, K$ be finitely generated  fully residually free
groups such that $ F\leq H\leq K$ and $K$ does not have a
sufficient splitting modulo $H$. Let  $D$ be an abelian JSJ
decomposition of $K$ modulo $H$ (which may be trivial).  There
exists a constant $N=N(K,H)$
 such that for each $F$-homomorphism
 $\phi:H \rightarrow F$ there are at most $N$
algebraic  pair-wise non-equivalent with respect to $\sim _{MAX}$, homomorphisms from
$K$ to $F$ that extend $\phi$.

 Moreover, if $H,K$ are as in Theorem \ref{th5}, the constant $N$ for the number of $\sim
 _{MAX}$-non-equivalent homomorphisms
 can be found effectively.
\end{theorem}
\begin{proof}  The statement about the existence of such constant $N$ is Theorem 3.5 \cite{Sela} (although there is no proof of Theorem 3.5 there).
We will show how to find this constant effectively.
To make presentation easier, we consider first the case when the group $K$ from the formulation of Theorem 11 does not have a splitting modulo $H$.
(in the terminology of \cite{Sela} it is a rigid limit group). We consider the formula
$$\exists P\exists Y_1,\ldots ,Y_m (\wedge _{i=1}^m S(P,Y_i)=1\wedge Y_i\neq Y_j (i\neq j)\wedge _{t=1}^k\wedge _{i=1}^m r_t(P,Y_i)\neq 1).$$
We know from Theorem 11 that possible number $m$ of algebraic solutions is bounded. Therefore for some  positive integer $m$ such a formula will be false. The minimal such $m$ can be found because the existential theory of $F$ is decidable. Therefore $N=m-1$.

Now we consider the case when the group $K$ has a splitting modulo $H$ but not a sufficient splitting ($K$ is solid in terminology of \cite{Sela}). This case is more complicated because we have to write  that  solutions corresponding to tuples  $Y_1,\ldots ,Y_m$  are not reducing and not in the same $\sim _{MAX}$ equivalence classes for $i\neq j$. This means that there exist no elements representing QH subgroups and no elements commuting with edge groups of the JSJ decomposition of $K$ modulo $H$ such that application of generalized fractional Dehn twists corresponding to these elements  take some of these solutions to reducing solutions or take one solution to the other.   This fact can be  expressed in terms of $\exists\forall$-sentence that is true if and only if there exists a homomorphism $H\rightarrow F$ that can be extended to $m$ algebraic and not $\sim _{MAX}$ equivalent homomorphisms $K\rightarrow F.$ The decidability of $\exists\forall$-theory of $F$ was proved in the previous section. Then the bound on $m$ can be found effectively because we can  find for which $m$ the sentence is false and therefore such homomorphism $H\rightarrow F$ does not exist.
\end{proof}

\newpage
\section{Quantifier elimination algorithm}\label{sec:6}
In this section we will prove Theorem \ref{main}. Consider the following formula

\begin{equation}\label{38neg} \Theta(P)=\exists Z \forall X\exists Y (U(A,P,Z,X,Y)=1\wedge V(A,P,Z,X,Y)\not =1),\end{equation}
where $A$ is a basis of $F=F(A)$\footnote{This formula $\Theta(P)$ is the negation of the formula $\Phi$ considered in \cite{KMel}.}. We will often skip $A$ in the formulas.

It was proved in \cite{Sela5},\cite{KMel} that every formula in the theory of a free group $F$ is equivalent to a boolean combination of $\exists\forall$-formulas. The general scheme of the proofs in \cite{Sela5} and in \cite{KMel} is quite similar: to use the so-called implicit function theorem or parametrization theorem that is \cite{Imp}, Theorem 12  (= existence of formula solutions
in the covering closure of a limit group) and to approximate any definable set and get its stratification  using certain verification process (based on the implicit function theorem) that stops after finite number of steps. But all the necessary technical results are proved differently (using actions on $\mathbb R$-trees in \cite{Sela5}, and using elimination process and free action of fully residually free groups on ${\mathbb Z}^n$-trees, which is equivalent to the existence of free length functions in ${\mathbb Z}^n$, in \cite{KMel}). The proof in \cite{KMel} is also algorithmic. It will be more convenient for us to follow our proof in \cite{KMel} and to use mostly our terminology. 

To obtain effective quantifier elimination to boolean combinations of $\exists\forall$-formulas it is enough to give an algorithm to find such a boolean combination that defines the set defined by  $\Theta(P).$

For every tuple of elements $\bar P$ for which $\Theta(\bar P)$ is true, there exists some $\bar Z$ and (by the Merzljakov theorem  \cite{KMel}, Theorem 4, a solution $Y=f(A,\bar P,\bar Z,X)$ of $U=1\wedge V\neq 1$ in $F(X)\ast F.$ All formula solutions of $U=1$ for all possible values of $P$ belong to a finite number of fundamental sequences with terminal groups $F _{R(U_{1,i})}\ast F(X),$ where $U_{1,i}=U_{1,i}(A,P,Z,Z^{(1)})$ and $F _{R(U_{1,i})}$ is a group with no sufficient splitting modulo $\langle A,P,Z\rangle$ (see Section 12.2, \cite{KMel}). These groups can be effectively found. 

We now consider each of these fundamental sequences separately.  Below we will not write the constants $A$ in the equations but assume that equations may contain constants. Those values $P,Z$ for which there exist a value of $X$ such that the equation 
$$V(P,Z,X,f(Z, Z^{(1)},P, X))=1$$ is satisfied for any function $f$ give a system of equations on $F _{R(U_{1,i})}\ast F(X).$ This system is equivalent to a finite subsystem (to one equation in the case when we consider formulas with constants).  Let $G$ be the coordinate group of this system and  $G_i, i\in J$ be the corresponding fully residually $F$ groups.

We introduced in  \cite[Section 12.2]{KMel}, the tree $T_{X}(G)$ which is constructed  the same way as $T_{EA}(\Phi)$ with $X, Y$ considered as variables and $P,Z, Z^{(1)}$ as parameters.  To each group $G_i$  we assign fundamental sequences modulo $\langle P,Z, Z^{(1)}\rangle$. Their terminal groups are groups  $F _{R(V_{2,i})}$, where
  $$V_{2,i}=V_{2,i}(P,Z,Z^{(1)}, Z_1^{(2)})$$
  that do not have a sufficient splitting modulo $\langle P,Z, Z^{(1)}\rangle$. Then we find all formula solutions $Y$ of the equation $$U(P,Z,X,Y)=1$$ in the corrective  extensions of the NTQ groups corresponding to these fundamental sequences for $X$ (see  \cite[Theorem 12]{Imp}). These formula solutions $Y$ are described by a finite number of fundamental sequences with terminal groups $F_{R(U_{2,i})},$ where $U_{2,i}=U_{2,i}(P,Z,Z^{(1)},Z_{1}^{(2)}, Z^{(2)}).$ Then again we investigate the values of $X$ that make the word $V(P,Z,X,Y)$ equal to the identity for all these formula solutions $Y$. And we continue the construction of $T_{X}(G).$ We can prove that this tree is finite exactly the same way as we proved the finiteness of the $\exists\forall$-tree.  We will call $T_{X}(\Theta)$ {\em the  parametric $\exists\forall$-tree} for the formula $\Theta(P)$. For each branch of the  tree $T_{X}$ we assign a
 sequence of groups $$F _{R(U_{1,i})},
F _{R(V_{2,i})}\ldots ,F _{R(V_{r,i})}, F _{R(U_{r,i})}$$
as in \cite[Section 12.2]{KMel}. Corresponding irreducible systems of equations are:

$$U_{1,i}=U_{1,i}(P,Z,Z^{(1)})=1,$$ which correspond to the terminal
groups of fundamental sequences describing $Y$  of level
$1$,
 $$V_{m,i}=V_{m,i}(P,Z,Z^{(1)}, Z_1^{(m)})=1,\ m=2,\ldots ,r$$
which correspond to the terminal groups of fundamental sequences describing $X$
 of level $m-1$ and 
$$U_{m,i}=U_{m,i}(P,Z,Z^{(1)},
Z_1^{(m)}, Z^{(m)})=1,\ m=2,\ldots ,r,$$ which correspond to the terminal
groups of fundamental sequences describing $Y$  of level
$m, m>1$, 
Systems $V_{m,i}=1$ correspond to vertices of $T_{X}$
that have distance $m$ to the root.

For each $m$ the group $F _{R(U_{m,i})}$ does not have a sufficient
splitting modulo the subgroup $\langle P,Z,Z^{(1)}, Z_1^{(m)}\rangle ,$ and the group
$F _{R(V_{m,i})}$ does not have a sufficient splitting modulo the
subgroup  $\langle P,Z,Z^{(1)}\rangle .$

On each step we consider terminal groups of
all levels. Below we will sometimes skip index $i$ and write $U_m=1,\
V_m=1$ instead of $U_{m,i}=1,\ V_{m,i}=1.$

\begin{prop}\label{cred}
A complete system of reducing quotients of a  limit group with no sufficient splitting modulo a subgroup can be found effectively.
\end{prop}
\begin{proof} Suppose first  that a limit group, denoted by $K$, does not have a  splitting modulo a subgroup $H$. For each reducing quotient $K_1$ of $K$ there is a discriminating family $\Psi$ of homomorphisms such that for every homomorphism $\psi\in\Psi$ the restriction $\psi _H$ can be extended  by infinitely many ways to  homomorphisms from $K_1$ to $F$. We can construct a system of equations $S=1$ such that $K=F_{R(S)}$. We can run the Elimination process described in \cite{JSJ} for $S=1$ modulo  $H$. Each homomorphism from  $H$ to $F$ can be extended by infinitely many ways to a solution of $S=1$ corresponding to a homomorphism from $K_1$ to $F$.  Among the obtained quotients of  $K$, we will obtain,  in particular, all the maximal reducing quotients of $K$ because the new splittings for quotients of $K$ will be seen in the Elimination process. Then we  have to compare the  reducing quotients that we obtain and take the maximal ones (using \cite{KMel}, Theorem 31).

Suppose now that  $K$ has a (non-sufficient) splitting modulo $H$. For each reducing quotient $K_1$ of $K$ there is a discriminating family $\Psi$ of homomorphisms such that for every homomorphism $\psi\in\Psi$ the restriction $\psi _H$ can be extended  by infinitely many ways to  homomorphisms from $K_1$ to $F$ which are minimal in their AEQ classes for $K$ (see \cite{KMel} for definition of AEQ class). We again can  take a system of equations $S=1$ in the free group $F$   such that $K=F_{R(S)}$.  We can run the Elimination process for $S=1$ modulo  $H$. Each homomorphism from  $H$ to $F$ can be extended by infinitely many ways to a solution of $S=1$ corresponding to a homomorphism from $K_1$ to $F$ that is AEQ-minimal as a homomorphism from $K$ to $F$.  We will obtain reducing quotients of $K$ and, in particular, all the maximal reducing quotients of $K$. Then we  have to compare the reducing quotients that we obtain and take the maximal ones.

\end{proof}
\subsection{Algorithm for the construction of the tree $T_X$.}
\begin{prop}\label{TX}

 There is an algorithm to construct the following: 
 \begin{enumerate}\item [1)] the finite parametric $\exists\forall$-tree $T_{X}$, \item [2)] for each branch of the tree $T_{X}$  the finite family of groups $$F _{R(U_{1,i})},
F _{R(V_{2,i})}\ldots , F_{R(V_{r,i})}, F_{R(U_{r,i})},$$ \item [3)]  for each vertex of the tree, a fundamental sequence describing either $Y$ (if the associated group is $F _{R(U_{j,i})}$) or $X$ (if the associated group is $F _{R(V_{j,i})}$).\end{enumerate}

\end{prop}
\begin{proof}

 The proof  uses the algorithm from Proposition \ref{cred} to construct a complete system of reducing quotients and the algorithm 
to construct fundamental sequences for a system of equations modulo a finite set of finitely generated subgroups. \end{proof}


\subsection{Configuration groups}

In order to show that a specialization $\bar P$ of the parameters $P$ is in the set $True(\Theta)$, one needs to find a specialization $\bar Z$ of variables $Z$ such that the corresponding $\forall\exists$-sentence
$$ \forall X\exists Y (U(\bar P,\bar Z,X,Y)=1\wedge V(\bar P,\bar Z,X,Y)\not =1)$$
 is true.  The proof that this sentence is true corresponds to a subtree of $T_X(\Theta)$ (we call it a true-subtree for $\bar P,\bar Z$). That is the proof consists of a finite sequence of formal solutions $Y$ in corrective extensions of corresponding block-NTQ groups in the vertices of this subtree.  Since $T_X(\Theta)$ is a finite tree, there is only a finite number of possibilities for  such true-subtrees of $T_X(\Theta)$.  Note that given $\bar P, \bar Z$ there may be several 
 true-subtrees for this tuple, but the number of possible true-subtrees is bounded. We will say that the sentence  associated with the tuple 
 $\bar P,\bar Z$  can be proved at level $m$ if the maximal depth of the true-subtree is $m$.
 
\begin{definition}
 Denote by $True(\Theta)_i$ the set of specializations $\bar P$ of $P$  for which there exists $\bar Z$ such that the corresponding $\forall\exists$-sentence can be proved on level $i$ (there are solutions to some $U_{i,j}=1$ but there are no remaining $X$ on level $i$ that is equivalent to the absence of solutions  to $V_{i+1,j}=1$).
\end{definition}

\begin{lemma} The set $True(\Theta)_1$ is an  $\exists\forall$-set, and there is an algorithm to write the formula defining this set.
\end{lemma}
\begin{proof}  We can write an $\exists\forall$- formula that says that there exists an index $i$, $Z, Z^{(1)}$ such that   $U_{1,i}(P,Z,Z^{(1)})=1$,
$Z^{(1)}$ is an algebraic solution, and for any index $j$ there is no specialization $Z_1^{(2)}$ such that $V_{2,j}=V_{2,j}(P,Z,Z^{(1)}, Z_1^{(2)})=1.$
\end{proof}

\begin{theorem}  \label{level2} The set $True(\Theta)_2$ is  in the Boolean algebra of  $\exists\forall$-sets, and there is an algorithm to find a  Boolean combination of $\exists\forall$-formulas defining $True(\Theta)_2$.
\end{theorem}  

The rest of the paper will be devoted to the proof of this theorem, and at the very end we will discuss the general case of 
$True(\Theta)_m$, where $m$ is bounded by the depth of the tree $T_X(\Theta).$\footnote
{ In Definition 27 and Definition 28, \cite{KMel} we define initial fundamental sequences of levels (2,1) and (2,2) and width $i$ (the possible width is bounded) modulo $P$.  In this paper we consider  both these levels as level 2. Since we are now considering the formula $\Theta$ such that $\Theta=\neg\Phi$ for the formula $\Phi$ considered in \cite{KMel}, we will slightly change the definition here. It will be more convenient to replace condition (6) from Definition 28 of \cite{KMel} by its negation and add this negation on level 2.}

\begin{definition}  \label{d2} Let $\bar P\in True(\Theta)_2$. Then there exists a family of  algebraic specializations $(\bar P,\bar Z, \bar Z^{(1)})$ for one of the groups $F _{R(U_{1,k})}$.  Let $F _{R(V_{2,1})},\ldots ,F _{R(V_{2,t})}$ be the whole family
of groups on level 1 of the proof constructed for this group $F _{R(U_{1,k})}$ ($k$ is fixed).  To construct the {\em initial
fundamental sequences of level} 2 and {\em width} $i=i_1+\ldots
+i_t$, we consider the fundamental sequences modulo the subgroup
$\langle P\rangle$ for the groups $H$ discriminated
by $i$ solutions $(\bar P,\bar Z,\bar Z^{(1)},\bar Z_1^{(2,j,s)},\bar Z^{(2,j,s)})$ of the systems
$$U_{2,m_s}(P,Z,Z^{(1)},Z_1^{(2,j,s)},Z^{(2,j,s)})=1,\ j=1,\ldots ,i_s,\ s=1,\ldots ,t,$$ with the following properties (for some $i_1,\ldots ,i_t$   such  solutions must exist):

(1) $(\bar P, \bar Z, \bar Z^{(1)})$ are algebraic solutions of $U_{1,k} (P,Z,Z^{(1)})=1$, 

\noindent
$(\bar P, \bar Z, \bar Z^{(1)},\bar Z_1^{(2,j,s)})$ are algebraic solutions of $V_{2,s} (P,Z,Z^{(1)}, Z_1^{(2)})=1$,  

\noindent
 $(\bar P, \bar Z, \bar Z^{(1)},\bar Z_1^{(2,j,s)},\bar Z^{(2,j,s)})$ are algebraic solutions of $U_{2,m_s} (P,Z,Z^{(1)}, Z_1^{(2)}, Z^{(2)})=1;$

(2) $ \bar Z_{1}^{(2,j,s)}$ are not $\sim_{MAX}$-equivalent to $\bar  Z_{1}^{(2,p,s)},
p\neq j,\ p,j=1,\ldots, i_s,\ s=1,\ldots ,t$;

(3) for any of the finite number of values of $Z_{1}^{(2)}$
the fundamental sequences for $V_{2,s}(P
,Z,Z^{(1)},Z_{1}^{(2)})=1$ are contained in the union
of the fundamental sequences for $U_{2, m_s}(P
,Z,Z^{(1)},Z_{1}^{(2,j)},Z^{(2,j)})=1$ for
different values of $Z^{(2,j,s)}$;

(4) there is no non-$\sim_{MAX}$-equivalent $\bar Z_{1}^{(2,i_{s}+1,s)}$, algebraic,
solving $V_{2,s}(P,Z,Z^{(1)}, Z_1^{(2)})=1, s=1,\ldots ,t$.

(5) the solution $(\bar P,\bar Z,\bar Z^{(1)})$  does not
satisfy a proper equation which for some $i\in J$ implies  $\overline V_i(\bar P,\bar Z, X)=1$  for any value of
$X$.

 (6)  for any $s$,  the solution $(\bar P
,\bar Z,\bar Z^{(1)},\bar Z_{1}^{(2,1,s)},\bar Z^{(2,1,s)})$  can not be
extended to a solution of some
$$V_{3,s}(P,Z,Z^{(1)},Z_{1}^{(2,1,s)},Z^{(2,1,s)},
Z_{1}^{(3,1,s)})=1.$$

We call this group $H$ a {\em configuration group}. We also call a tuple $$(\bar P, \bar Z,\bar Z^{(1)},\bar Z_{1}^{(2,j,s)},\bar Z^{(2,j,s)},\ j=1,\ldots i_s,\ s=1,\ldots ,t)$$ satisfying the conditions above {\bf a certificate} for $\Theta$ for $\bar P$ (of level 2  and width $i$). 
 We add to the generators of the configuration group additional variables $Q$ for the primitive roots of a fixed set of elements for each certificate (these are primitive roots of the images in $F$ of the edge groups and abelian vertex groups in the relative JSJ decompositions of the groups $F _{R(V_{2,1})},\ldots ,F _{R(V_{2,t})}$). \end{definition}
 
 When we consider the initial fundamental sequence of level 2 and width $i=i_1+\ldots +i_t$, we have fixed the group
 $F _{R(U_{1,k})}$, the number $i_s$ of {Max}-equivalence classes for each of the groups $F _{R(V_{2,s})}$, and the number of {Max}-equivalence classes for each $F _{R(U_{2,m_s})}$.  It follows from Theorem \ref{term} that there is only a finite number of  initial fundamental sequences of level 2. 
 
 Each group $H$ from this definition is a fundamental group of some system of equations, say $$W(P,Z,Z^{(1)},Z_1^{(2,j,s)},Z^{(2,j,s)}, Q, \ j=1,\ldots i_s,\ s=1,\ldots ,t)=1.$$

 Some generic families of specializations $(\bar P,\bar Z,\bar Z^{(1)},\bar Z_1^{(2,j,s)},\bar Z^{(2,j,s)}, \bar Q, \ j=1,\ldots i_s,\ s=1,\ldots ,t)$ that factor through the fixed  initial completed
fundamental sequence $f$ (modulo $\langle P\rangle $, with terminal group $Term(P, P^{(1)})$) of level 2 and width $i$, may fail to be certificates for $\bar P$ in one of the following ways:

1.  Specializations in a generic family for $f$ may not satisfy property (1),  instead of being algebraic they may become reducing, or two not Max-equivalent specializations may become equivalent (failing (2)). These specializations factor through a finite number of  limit groups $Red_1,\ldots ,Red _k$ (extra variables may be added to demonstrate this). By the Implicit Function Theorem the NTQ groups for them are corrective extensions of the NTQ group for $f$. A certificate for $f$ does not factor through any of 
$Red_1,\ldots ,Red _k$.

2. Specializations for which property (3) fails  factor through a finite number of limit groups with additional equations involving root variables $Q$, $RQ_1,\ldots ,RQ_m$.

3.  Specializations may not satisfy property (6), in this case they can be extended to  specializations that  factors through a finite number of limit groups of higher level $F _{R(V_{3,m_s})}$.

Notice that all the specializations that can be extended to specializations that factor through these limit groups $Red_1,\ldots ,Red _k$,  $F _{R(V_{3,m_s})}$, $RQ_1,\ldots ,RQ_m$,  are not certificates.

4. A generic family may fail property (4) in another way. For tuples $\bar P,\bar Z,\bar Z^{(1)}$ there may be more  non-$\sim_{MAX}$- equivalent solutions of the equation $V_{2,s}(P,Z,Z^{(1)}, Z_1^{(2)})=1$ than $i_s$. These tuples then factor through one of the groups $H_{surplus}$ discriminated by solutions of $W=1$ together with solutions $Z_1^{(2,i_s+1,s)}\rightarrow F$ minimal with respect to fractional Dehn twists.  But some tuples for $f$ factoring through $H_{surplus}$ may still be certificates provided they are  going through one of the fundamental sequences for which value of $Z_1^{(2,i_s+1,s)}$ is either reducing or $\sim_{MAX}$-equivalent to one of $Z_1^{(2,j,s)}, j=1, \ldots ,i_s.$

This analysis shows the following
\begin{prop}For each initial completed
fundamental sequence of level 2 and width $i$,  for each  $\bar P$ factoring through this fundamental sequence for which there exists a certificate, there are the following possibilities: 
 \begin{enumerate}\item  there exists 
 a generic family of certificates (corresponding to the fundamental sequence), 
 \item any certificate in this fundamental sequence can be extended by $Z_1^{(2,i_s+1,s)}$ so that the whole tuple factors through one of the groups $H_{surplus}$, but these variables  $Z_1^{(2,i_s+1,s)}$  factor through one of the fundamental sequences for which $Z_1^{(2,i_s+1,s)}$  is either reducing or $\sim_{MAX}$-equivalent to one of $Z_1^{(2,j,s)}, j=1, \ldots ,i_s.$\end{enumerate}
In the former case we say that the fundamental sequence has depth 1, it the latter case
we will consider fundamental sequences of depth 2  (for level 2 and width $i$).  Similarly there may be constructed fundamental
sequences of larger depth (for level $2$ and width $i$).
\end{prop} 

Completed fundamental sequences for $H_{surplus}$ correspond to  corrective extensions of the NTQ group for the completed fundamental sequence $f$,  and  have terminal groups with no sufficient splitting modulo $Term(P,P^{(1)})$ that we denote by $Term_1(P,P^{(1)},P^{(2)})$. 
 
 Notice that we do not know a system of equations defining  a configuration group. We, therefore, need the following result.

\begin{prop} \label{tight1} Let $H=F _{R(W)}$ be one of the configuration groups with generators
$$P, Z, Z^{(1)},
Z_{1}^{(2,j,s)},Z^{(2,j,s)}, Q, \  j=1,\ldots ,i_s,\ s=1,\ldots
,t. $$    Then there is an algorithm to find  each terminal group of each fundamental sequence  for $H$ modulo $P$.
\end{prop}

\begin{proof}  As in the proof in  of \cite [Theorem 11]{KMel}, we extensively use the technique of {\em generalized
equations}  described in \cite[Subsection 4.3 and Section 5]{Imp},  and {\em cut equations} described in  \cite[Section 5.7]{Imp}.
The reader has to be familiar with these sections of \cite{Imp}. In the proof of  \cite [Theorem 11]{KMel} we show how to construct, given a group $K$ that does not have a sufficient splitting modulo a subgroup $H$, a finite system of generalized equations $\Pi$ (see \cite[Section 7.7]{Imp}) for a minimal in its Max-class solution such that the intervals of $\Pi$ are labeled by values of the generators of $H$. For each system $$U_{2,m_s}(P,Z,Z^{(1)},Z_{1}^{(2,j,s)},Z^{(2,j,s)})=1$$ we construct a generalized equation modulo the parametric subgroup \newline $\langle P,Z,Z^{(1)},Z_{1}^{(2,j,s)}\rangle $. Since solutions are minimal algebraic, we can move all bases to the  intervals of this generalized equation labeled by $P,Z,Z^{(1)},Z_{1}^{(2,j,s)}$.  
For the intervals labeled by $P,Z,Z^{(1)}$ we add a generalized equation for the system
 $$V_{2,m}(P,Z,Z^{(1)}, Z_{1}^{(2,j,s)})=1$$ modulo the parametric subgroup $\langle P,Z,Z^{(1)}\rangle .$
 For the intervals labeled by $P,Z$ we add  generalized equations for the system
 $$U_{1,m}(P,Z,Z^{(1)})=1$$ modulo the parametric subgroup $\langle P,Z\rangle .$ The intervals labeled by $P$ will be the same for all these generalized equations. Similarly we identify all the intervals labeled by the same variables that occur in different generalized equations.
We now add to the obtained generalized equation,  the inequalities that guarantee that conditions (1)-(6) are satisfied. These inequalities are just indicating that
specializations of variables corresponding to some sub-intervals of the generalized equation must not be identities. But this is a standard requirement  for a solution of a generalized equation. For example, we write an equation $r_1(Z_{1}^{(2,j)})=\lambda _1$ and set that $\lambda _1$ is a base of the generalized equation. Then the condition $\lambda _1\neq 1$ must be automatically satisfied for a solution of a generalized equation.  So we can construct a finite number of generalized equations (denote this set $\mathcal {GE}$)
 such that each certificate corresponding to minimal in their Max-classes specializations is a solution of one of these generalized equations $\mathcal {GE}$.

We now construct fundamental sequences of solutions of the equations $\mathcal {GE}$ modulo $\langle P\rangle $.  Notice that not all solutions from the fundamental sequence
satisfy the necessary inequalities, but if we restrict the sets of automorphisms on all the levels to those whose application preserves corresponding generalized equations, we will have solutions of inequalities too. Therefore, a generic family of solutions does satisfy the inequalities. Using our standard procedure we construct fundamental sequences induced by the subgroup with generators $$P, Z, Z^{(1)},
Z_{1}^{(2,j,s)},Z^{(2,j,s)}, Q,\  j=1,\ldots ,i_s,\ s=1,\ldots
,t.$$  The subgroups generated by the images of these generators in the terminal groups of these fundamental sequences are precisely the terminal groups of the fundamental sequences for $H$ modulo $P$.

\end{proof}

This implies the following result.

\begin{cor} \label{tight2} There is an algorithm to construct the initial fundamental sequences for $Z$ of level 2 and width $i$. \end{cor}

For a given value $\bar P$ of $P$ the formula $\Theta$
can be proved on level  2 and depth 1
 if and only if  the following
conditions are satisfied.
\begin{enumerate}

\item [(a)] There exist algebraic solutions  for some system of equations
$U_{i,coeff}=1$ corresponding to  the terminal group of a fundamental
sequence $V_{i,\rm fund}$ for a configuration group modulo $P$.
\item [(b)] These solutions do not factor through the fundamental sequences that
describe solutions from $V_{i,\rm fund}$ that do not satisfy one of
the properties (1)-(6). There is a finite number of such fundamental
sequences, they are fundamental sequences corresponding to groups $Red_1,\ldots ,Red _k$,  $F _{R(V_{3,m_s})}$, $RQ_1,\ldots ,RQ_m$ and $H_{surplus}$. Therefore $\bar P$ cannot be extended to algebraic solutions of equations corresponding to terminal groups of these fundamental sequences. 

\item  [(c)] These solutions do not factor through 
 the terminal groups of fundamental sequences of
level 2 and greater depth derived from $V_{i,\rm fund}.$

\end{enumerate}

In this case there is a generic certificate of level $2$ width $i$ and depth 1.
These conditions can be described by a boolean combination of  $\exists\forall$-formulas of type (\ref{AE}).
Similarly we consider fundamental sequences of level $2$ width $i$ and depth 2 and deeper fundamental sequences of level 2 width $i$.  The main technical tool that is needed to show that the procedure of constructing deeper and deeper fundamental sequences of level 2 and width $i$ stops, is the  notion of so called tight enveloping NTQ groups and fundamental sequences.

\subsection{Tight Enveloping NTQ Groups}
We will now describe the construction of  a {\em tight enveloping NTQ group and fundamental sequence}. As we just mentioned,  tight enveloping NTQ groups serve as the main technical tool that is needed to show that the procedure of constructing deeper and deeper fundamental sequences of level 2 and width $i$ stops. This construction is needed to control the complexity of fundamental sequences. 
We begin with a fundamental sequence satisfying first and second restrictions and the NTQ group for it that we denote $F _{R(L_1)}.$  We denote the fundamental sequence $c(L_1)$. Let $G=F_{R(U)}$ be a subgroup of $F _{R(L_1)}$  and suppose we constructed the induced NTQ group for $G$, $Ind(F _{R(U)})$ as described in Subsection \ref{inher}. Our goal is to construct
a fundamental  sequence $c(U)$ and an NTQ group for $G$ such that the Kurosh rank of the NTQ group for $G$ with respect to $c(L_1)$ is the same as the Kurosh rank of $c(U)$. 
\begin{enumerate}\item [(a)] We take the NTQ group induced by the image  of $G$ in $F _{R(L_1)}$  from
$F _{R(L_1)}$,  this does not
increase the Kurosh rank, because we
 add only elements from abelian subgroups and conjugating elements that are mapped to the identity on the next lower level. Denote this group by $Ind(F _{R(U)})$. It is an NTQ group, mappings between different levels are restrictions of the mappings for $F _{R(L_1)}$. Then we do the
following.

  \item [(b)] We add from the top to the bottom (considering on level $i+1$ the image of the group
extended on level $i$)  those QH subgroups $Q$ of the  group $F _{R(L_1)}$ that intersect
  the image on level $i+1$ in a subgroup of finite index (in $Q$) and have less free variables than the
  subgroup in the intersection.
  
  \item [(c)] We add edge groups of abelian subgroups of the
enveloping group that have non-trivial intersection with
$Ind(F _{R(U)})$ if this does not increase the Kurosh rank.

\item [(d)] We add to $Ind(F_{R(U)})$ from bottom to top all the QH
subgroups of $F _{R(L_1)}$ that have non-trivial intersection with
$Ind(F _{R(U)})$, do not have free variables, the corresponding level
of $Ind(F _{R(U)})$ (more precisely, the extension of $Ind(F _{R(U)})$ on this step) intersects non-trivially some of their adjacent
vertex groups, and their  addition decreases the Kurosh rank.

\item [(e)] We add from bottom to top all the elements that conjugate different $QH$
subgroups (abelian vertex groups) of $Ind(F _{R(U)})$ into the same
QH subgroup of $F_{R( L_1)}$ if this decreases the Kurosh rank.

\item [(f)]  We add (from bottom to top) to each non-cyclic  factor $H_i$ in  the free decomposition $H_1\ast\ldots\ast H_t\ast F$ of the image of the constructed fundamental sequence on each level,  the abelian vertex groups and  edge groups on this level of $c(L_1)$ that are intersected non-trivially by $H_i$.
\end{enumerate}

We make these steps (which we call adjustment) iteratively and denote the obtained group by $Adj (F _{R(U)})$.
We  repeat the adjustment iteratively as many times as needed.
We call the constructed NTQ group the tight
enveloping NTQ group. We will also call the corresponding system
(fundamental sequence) the tight enveloping system (fundamental
sequence) and denote this fundamental sequence $c(U)$ by $TEnv
(G; L_1)$.  It has  other important properties that we will mention later.

 Given a fully residually free group $G=F _{R(U)}$, the  NTQ system $W(X_1,\ldots ,X_n)=1$ corresponding to a fundamental sequence for $U=1$ (with the quadratic system $S_1(X_1,\ldots, X_n)=1$ corresponding to the top level and its image), a system of equations ${\mathcal P}=1$  with coefficients in $F _{R(W)}$ having a solution in some extension of $F _{R(W)}$, we  construct fundamental sequences (satisfying first and second restrictions) for ${\mathcal P}=1$ modulo the non-cyclic free factors of the second level $\langle X_2,\ldots ,X_n\rangle $ of $F _{R(W)}.$ Consider one of these fundamental sequences and construct the NTQ group  for it.  Denote it $F _{R(L_1)}$.  \footnote 
 {This construction is similar to \cite{KMel}, Section 11.2, but there is an error  in Section 11.2 (that is easy to fix).  We erroneously said that the fundamental sequences above should be constructed modulo non-cyclic free factors of $F_{R(U)}$. Instead, they have to be constructed  modulo  non-cyclic free factors of the second level $\langle X_2,\ldots ,X_n\rangle $ of $F _{R(W)}.$} 

 \begin{prop}\label{tight}
1) Given a fully residually free group $G=F _{R(U)}$, the canonical NTQ system $W=1$ corresponding to a branch of the canonical embedding tree $T_{CE}(F _{R(U)})$ of the system $U=1$, a system of equations ${\mathcal P}=1$  with coefficients in $F _{R(W)}$ having a solution in some extension of $F _{R(W)}$, there is an algorithm for the construction of tight enveloping NTQ groups and fundamental sequences $TEnv
(G; L_1)$.

2) The bound in   \cite[Lemma 28]{KMel} can be found effectively.
\end{prop}
\begin{proof} 1)  Indeed, in the construction of tight enveloping fundamental sequences and systems we have to solve the following algorithmic problems:
find intersection of conjugates of QH-subgroups of two NTQ groups  and 
find solution sets of quadratic systems of equations in NTQ groups (to determine the rank of a QH subgroup). These problems were solved algorithmically in \cite{KMRS} and \cite{JSJ}.

2) The bound in  \cite[Lemma 28]{KMel} can be found effectively as in Theorem \ref{term}.
\end{proof}

 \subsection{Sketch of the proof of Theorem 8}
 Now we will very briefly recall how it is shown in \cite{KMel} that the procedure for constructing deeper and deeper fundamental sequences of level 2 and width $i$ stops. The goal of this explanation is to show that  Propositions \ref{tight1}, \ref{tight} and Corollary \ref{tight2} is all what is needed to make this procedure algorithmic.   Corollary \ref{tight2} takes care of the initial step of the procedure. For a fundamental sequence of level 2, width $i$, and depth $n$ with block-NTQ group $N_n$ constructed on step $n$ of the procedure, we consider  groups discriminated by this sequence and $\sim _{MAX}$-minimal values of extra variables $Z_1^{(2,i_s+1,s)}$ such that 
 $$V_{2}(P,Z,Z^{(1)}, Z_{1}^{(2,i_s+1,s)})=1,$$ 
  and values of $Z_1^{(2,i_s+1,s)}$ are either reducing or $\sim _{MAX}$-equivalent to the value of one of $Z_1^{(2,j,s)},$ $j=1,\ldots ,i_s$. 
It is proved that such a group  is either a proper quotient of a corrective extension $N_{n(corr)}$ of $N_n$ or a proper quotient of the  group obtained from $N_{n(corr)}$ by adding roots of some elements. 
We construct fundamental sequences for these groups. Then we construct  tight enveloping fundamental sequences for $N_1$ extracted from these fundamental sequences  and construct for them block-NTQ groups $N_{n+1}$.
 This is  step $n+1$ of the procedure, and $N_{n+1}$ has depth $n+1$.  
The main idea is that  these deeper fundamental sequences of level 2 and width $i$ can be constructed in such a way that  the complexity of the fundamental sequence of depth $n+1$  is not larger than the complexity of the fundamental sequence of depth $n$,  and the complexities cannot be the same on more than the bounded number of steps (we can find this bound algorithmically by Proposition \ref{tight}, part 2).
 
The non-increase in the complexity is organized as follows. As the size of  a
QH subgroup $Q$ in the tight enveloping NTQ group we consider the
size of the QH subgroup in the enveloping group containing $Q$ as a
subgroup of finite index.
 With each QH subgroup $Q$ corresponding to the equation $S_1(X_1,\ldots, X_n)=1$ we can associate a punctured surface. A simple closed curve on the surface corresponds to an element in $Q$. We suppose that $G$ is embedded into $F _{R(L_1)}$.  We suppose also that the family of simple closed curves (on the surface corresponding to $Q$) that are mapped to the identity  on the next level of the fundamental sequence for $W(X_1,\ldots ,X_n)=1$, is compatible with the family of splittings induced on  $Q$  from its images along the
 fundamental sequence for $L_1=1$ ($Q$ inherits a sequence of splittings from decompositions corresponding to $F _{R(L_1)}$ in which the boundary elements of $Q$ are elliptic and non-trivial).

These  compatibility   conditions and the construction of  the tight enveloping fundamental sequence  imply that the Kurosh rank of   $TEnv
(G; L_1)$ extracted from the fundamental sequence for $L_1$ is less than or equal to the Kurosh rank of the fundamental sequence for ${S}_1=1$
modulo free factors of $\langle X_2,\ldots ,X_n\rangle$. If the Kurosh ranks
are the same, we  reorganize the levels of the enveloping
system ${ L}_1=1$ moving down stable QH subgroups (see \cite[Section 7.3]{KMel}) into another system $L_2=1$ (and well aligned fundamental sequence)  so that they have
the same solutions obtained from the fundamental sequences, any homomorphism from $G$ to $F$ factoring through a fundamental sequence for $F _{R(W)}$ that factors through the fundamental sequence for $L_1$, also factors through the fundamental sequence for $L_2$, and  $size (TEnv(G; L_2))\leq size(S_1)$ for the tight enveloping fundamental sequence for $G$ induced from the the fundamental sequence for $L_2$.
If all the parameters (Kurosh rank, size, abelian rank, rank) for $S_1=1$ and $TEnv(G; L_2)$ are the same, then
$TEnv(G; L_2)$ has one level, and the abelian decomposition has the same graph and QH and abelian vertex groups corresponding to the system  $S_1=1$.  Notice, that  the
Kurosh rank of the tight enveloping fundamental sequence $TEnv
(G; L_2)$ is the same as the maximal Kurosh rank of the corresponding
subgroup in the terminal group in the enveloping fundamental
sequence modulo free non-cyclic factors of $\langle X_2,\ldots ,X_n\rangle$. (Notice also that we do not induce the fundamental sequence by the terminal free factor of $S_1=1$ (generated by free variables of quadratic equations in $S_1=1$), we just take its image in $F _{R(L_2)}.$)

Suppose now that $H\leq F _{R(W)},$ $L_1=1$ the same as above, and $TEnv(H; S_1)$ is a tight enveloping fundamental sequence for $H$.  Then one can modify the system $L_1=1$ into $L_2=1$ (with the well aligned fundamental sequence)
with the following properties: 

1) they have
the same fundamental solutions, 

2) any homomorphism from $H$ to $F$ factoring through the fundamental sequence for $F _{R(W)}$ that factors through the fundamental sequence for $L_1=1$ also factors through the fundamental sequence for $L_2=1$,

3)  one can construct a tight enveloping fundamental sequence $TEnv(H; L_2)$ for $H$ such that the Kurosh rank of 
$TEnv(H; L_2)$  is bounded by the Kurosh rank of $TEnv
(H; S_1)$  and in the case of equality  and $(size, ab)$ for $TEnv(H; L_2)$ is bounded  
by the $(size, ab)$ for $TEnv(H; S_1)$.  

In the case of the equality $TEnv(H; L_2)$ has one level, and the abelian decomposition is similar to the decomposition for $TEnv(H; S_1).$

The procedure for the construction of fundamental sequences of level 2 stops after finite number of steps and is algorithmic. This proves Theorem \ref{level2}. The set $True(\Theta)_2$ is  defined by  a Boolean combination of  $\exists\forall$-sets, and there is an algorithm to find this Boolean combination. 

We consider fundamental sequences of all levels  $m$ similarly. We now can make all the steps of the quantifier elimination procedure (to boolean combination of formulas (\ref{AE})) algorithmically.

This proves Theorem \ref{main}.

\end{document}